\documentclass[a4paper,12pt]{amsart}

\usepackage{amsthm}
\usepackage{amsmath}
\usepackage{amssymb}
\usepackage{latexsym}
\usepackage{enumerate}

\usepackage{graphicx}

\usepackage{bm}

\renewcommand{\thetheoremName}

\newtheorem{theorem}{Theorem}[section]

\newtheorem{proposition}[theorem]{Proposition}
\newtheorem{corollary}[theorem]{Corollary}

\theoremstyle{definition}
\newtheorem{definition}[theorem]{Definition}

\newtheorem{remark}[theorem]{Remark}

\numberwithin{equation}{section}

\newcommand{\Jac}{\operatorname{Jac}}
\newcommand{\spanning}{\operatorname{Span}}
\newcommand{\Jv}{\mathcal{J}}
\newcommand{\dist}{\operatorname{dist}}

\newcommand{\sign}{\operatorname{sign}}

\newcommand{\Le}{\operatorname{Le}}
\newcommand{\inj}{\operatorname{inj}}

\newcommand{\LL}{\mathcal{L}}

\newcommand{\A}{\mathcal{A}}
\newcommand{\arsinh}{\operatorname{arsinh}}

\newcommand{\arccosh}{\mathrm{arccosh} \,}

\begin{document}

\title[Tractors and tractrices]
{Tractors and tractrices\\
in Riemannian Manifolds}

\author{Jesper J. Madsen}
\address{Department of Chemistry, The University of Chicago, Illinois, USA}
\email[J. J. Madsen]{jjmadsen@uchicago.edu}

\author{Steen Markvorsen}
\address{DTU Compute, Mathematics, Kgs. Lyngby, Denmark}
\email[S. Markvorsen]{stema@dtu.dk}



\begin{abstract}
We generalize the notion of planar bicycle tracks -- a.k.a. one-trailer systems -- to so-called tractor/tractrix systems in general Riemannian manifolds and prove explicit expressions for the length of the ensuing tractrices and for the area of the domains that are swept out by any given tractor/tractrix system. These expressions are sensitive to the curvatures of the ambient Riemannian manifold, and we prove explicit estimates for them based on Rauch's and Toponogov's comparison theorems. Moreover, the general length shortening property of tractor/tractrix systems is used to generate geodesics in homotopy classes of curves in the ambient manifold.
\end{abstract}

\maketitle


\section{Introduction}  \label{secIntroI}
The classical tractrix curve appears in virtually every textbook on differential geometry of curves and surfaces -- see e.g. \cite{doCarmo1976}, \cite{Gray2006}, \cite[p. 239]{Spivak3}, \cite[p. 67]{Klingenberg1978},  and figure \ref{fig2DTractorClassic1} below. The tractrix has a long and fascinating history beginning with works of Huygens, Leibniz, Newton, and Euler, see \cite{Cady1965}. Not to mention the celebrated watch track experiment by Claude Perrault, long before the invention of the bicycle -- see \cite{foote2013a}, \cite{dunbar2001a}, \cite{figliolini2015a}.\\

We quote from \cite[p. 1065]{Cady1965}:``At a meeting in Paris in 1693 Claude Perrault laid his watch on the table, with the long chain drawn out in a straight line (\cite[vol. 3]{Cantor1965}). He showed that when he moved the end of the chain along a straight line, keeping the chain taut, the watch was dragged along a certain curve. This was one of the early demonstrations of the tractrix.''\\

Moreover, this classical tractrix gives rise to interesting classical surfaces as well: When rotated around the axis along which the tractrix is pulled, any regular segment of the tractrix curve generates a pseudo-sphere of constant negative Gauss curvature, and the involute of the full classical tractrix curve, including its cusp singularity, is a catenary, which itself, when rotated around the axis, generates a catenoid, i.e. a minimal surface -- the only nontrivial minimal surface of rotation. Further, the generalization of the straight line chain to a straight line rigid wagon pole in $\mathbb{R}^{3}$ pulled or pushed along a given tractor curve in space will sweep out a ruled surface, which -- by our definition of the generalized tractor/tractrix systems below -- is a tangent-developable surface (with the tractrix curve as its striction curve), hence it is flat with Gauss curvature $0$, see figure \ref{figAreaSweep2}. It follows that the tractor/tractrix systems in $\mathbb{R}^{3}$ can be studied via their canonical isometric representations in the plane.\\

The directly related geometry of bicycle tire tracks in the plane has been extensively studied -- also recently -- again in the generalized setting of any given front wheel (tractor) track with a following back wheel (tractrix) track,  see \cite{tabachnikov2006a}.
The bicycle systems have interesting modern applied ramifications and generalizations in motion planning, $n$-trailer systems, robotics and in non-holonomic multisteering systems, see
\cite{rouchon2008a}, \cite{Jakubiak2014}, \cite{li2011a}, \cite{tilbury1995a}.\\

We also find intriguing analogues in such diverse fields as floating bodies in equilibrium, see \cite{wegner2003} and \cite{wegner2007}, electron trajectories in parabolic magnetic fields and  Schr\"{o}dinger's equation, see \cite{levi2014} and \cite{levi2017a}, as well as to the inner workings of the so-called Prytz planimeter (used for area measurements of planar domains), see \cite{foote2013a} and \cite{levi2009a}, and to obtain isoperimetric inequalities for wave fronts as discussed in  \cite{howe2011a}.\\

Moreover, the tractor/tractrix problem  carries direct connotations to the so-called pursuit problems -- see e.g. the interesting paper \cite{bruckstein1993a} which is motivated by ant experiments performed by R. P. Feynman, reported in \cite[p. 79]{Feynman1985}.\\

As already alluded to in the abstract, our primary concern in the present paper is to define -- and to make an initial study of -- the most natural generalization of bicycle track systems in the plane to
 tractor/tractrix systems in  Riemannian manifolds. The corresponding  motion planning problems as well as a wealth of other applications -- like the ones indicated above -- can easily be formulated in this general context. The solutions to such problems are bound to be of interest, both in differential geometry and in the
respective applied fields.


\subsection{Outline of paper}
In the following section we first define the notion of tractor/tractrix systems in Riemannian manifolds. We illustrate the general setting in 2D in figures \ref{fig2DTractorContract1} and \ref{fig2DTractorContract2} and on surfaces in 3D in figures \ref{figWildTrack1}--\ref{figWildTrack4}. They give a first glimpse of our main results, which are concerned with the length-shortening property of tractrices and with estimating the area that is swept out in between a tractor and a tractrix. In section \ref{sec3D} we derive an ODE system which is equivalent to any given tractor/tractrix system in space. We illustrate specific solutions in 3D for the cases where the tractor is a member of a certain family of helices.  In section \ref{sec2D} we recover the classical tractrix in the plane -- and find the explicit formulas for the distance of the tractrix to the tractor, the curvature and the total curvature of the tractrix, and the explicit area that is swept out by the wagon pole connecting the tractrix to the tractor. Our first main result is theorem \ref{thmMainLA} in section \ref{secMainLA}, which gives explicit expressions for the length of tractrices and for the sweeping area in the most general case. Since both expressions depend on the Jacobi fields generated by the (geodesic) wagon pole motion, we use Rauch's comparison theorems -- reviewed in section \ref{secJacobiField} -- to estimate the tractrix length and the ensuing swept out area via curvature bounds on the ambient manifold. In section \ref{secHomotopy} we use the length contraction property of the general tractor/tractrix system to find shortest curves in homotopy classes of curves in the given manifold. In the final sections of the paper we are concerned with the special tractor/tractrix systems where the tractor is a geodesic in the manifold. We show in section \ref{secVariableToponogov} how Toponogov's triangle comparison theorems can be used in these cases to estimate the distance of the tractrix to the tractor as well as to estimate the curvature of the tractrix. This is obtained via comparison with the corresponding tractor/tractrix systems in constant curvature ambient spaces, which themselves are analyzed explicitly in sections \ref{secSpaceFormPull} and \ref{secLongPoles}. The final section \ref{secDiscurs} contains a brief
discussion on possible future work on tractor/tractrix systems -- in the respective applied fields as well as in differential geometry.


\section{The Riemannian tractor/tractrix systems}  \label{secRiemannSetting}
We let $(M^{n}, g)$ denote a complete Riemannian manifold with dimension $n$ and metric $g$. A tractor/tractrix system in $(M^{n}, g)$ is then defined via the following ingredients:

\begin{definition}
A {\textit{tractor curve}} $\eta(t)$ is a given smooth piecewise regular curve in $(M^{n}, g)$.  In this work we shall mainly (but not exclusively)
consider regular tractors. The parameter $t$ may or may not be an arc length parameter of $\eta$.  A {\textit{pulled tractrix curve}} $\gamma(s)$ is any smooth piecewise regular curve with the following defining property at every regular point: For every $s$ the unit speed geodesic $\lambda_{s}(u)$ issuing from  $\gamma(s)$ in the direction of $\gamma^{\prime}(s)$ has $\lambda_{s}(\ell) =  \eta(t(s))$ for some value of $t=t(s)$ and for some fixed constant $\ell > 0$. A {\textit{pushed tractrix curve}} is defined similarly using
the opposite direction $-\gamma^{\prime}(s)$ for the construction of $\lambda_{s}(\ell)$. In each case $\lambda_{s}(u)$, $u \in [0, \ell]$, is called the instantaneous {\textit{(wagon) pole}} by which the tractrix $\gamma(s)$ is pulled or pushed by the tractor $\eta(t(s))$.
\end{definition}

\begin{remark}
A wagon pole can -- in principle -- have any length $\ell$ beyond any of its instantaneous cut- and conjugate loci of the initial point $\lambda_{s}(0)$. In section \ref{secLongPoles} we briefly analyze the setting of tractor/tractrix systems on the sphere with a geodesic tractor and poles of any length.  However, in order to apply results from comparison geometry in sections \ref{secSpaceFormPull} and \ref{secVariableToponogov} we shall usually assume, that $\ell$ is not too large in comparison with the curvatures of the ambient space $(M, g)$.
\end{remark}

We first show a typical example of two non-trivial tractor/tractrix systems in the plane in figures \ref{fig2DTractorContract1} and  \ref{fig2DTractorContract2}. They are related in the sense that the tractrix in the first figure is used as the tractor curve in the second figure. In both figures the blue tractrix curve is clearly shorter than the respective red tractor curve. This is a typical phenomenon. We show below that it holds true in any Riemannian manifold except in those trivial cases where the tractrix itself (and thence also the tractor curve) is already a geodesic, see corollary \ref{corLLdiffer}.

\begin{figure}[h!]
\begin{center}
\includegraphics[width=60mm]{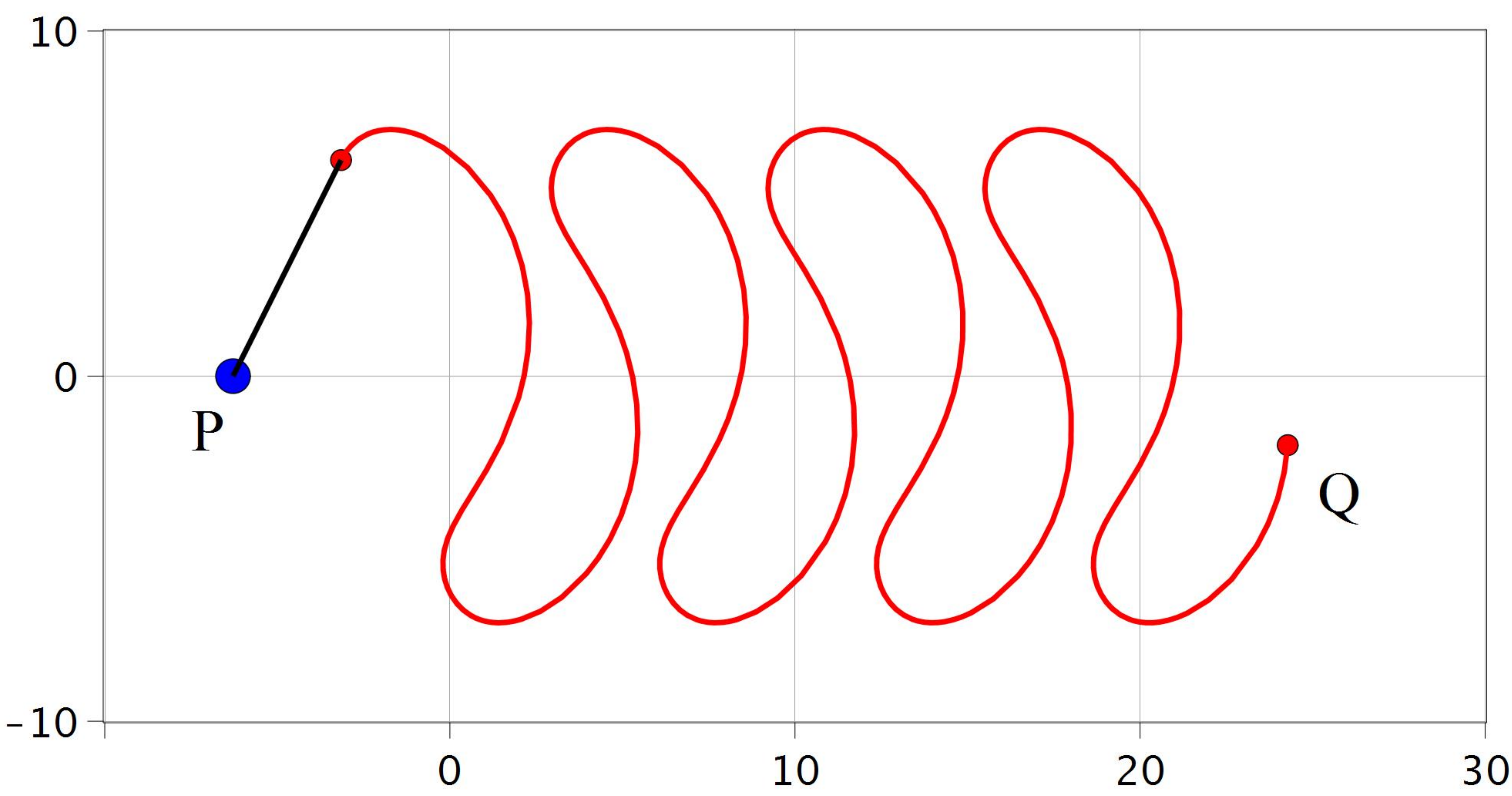}\includegraphics[width=60mm]{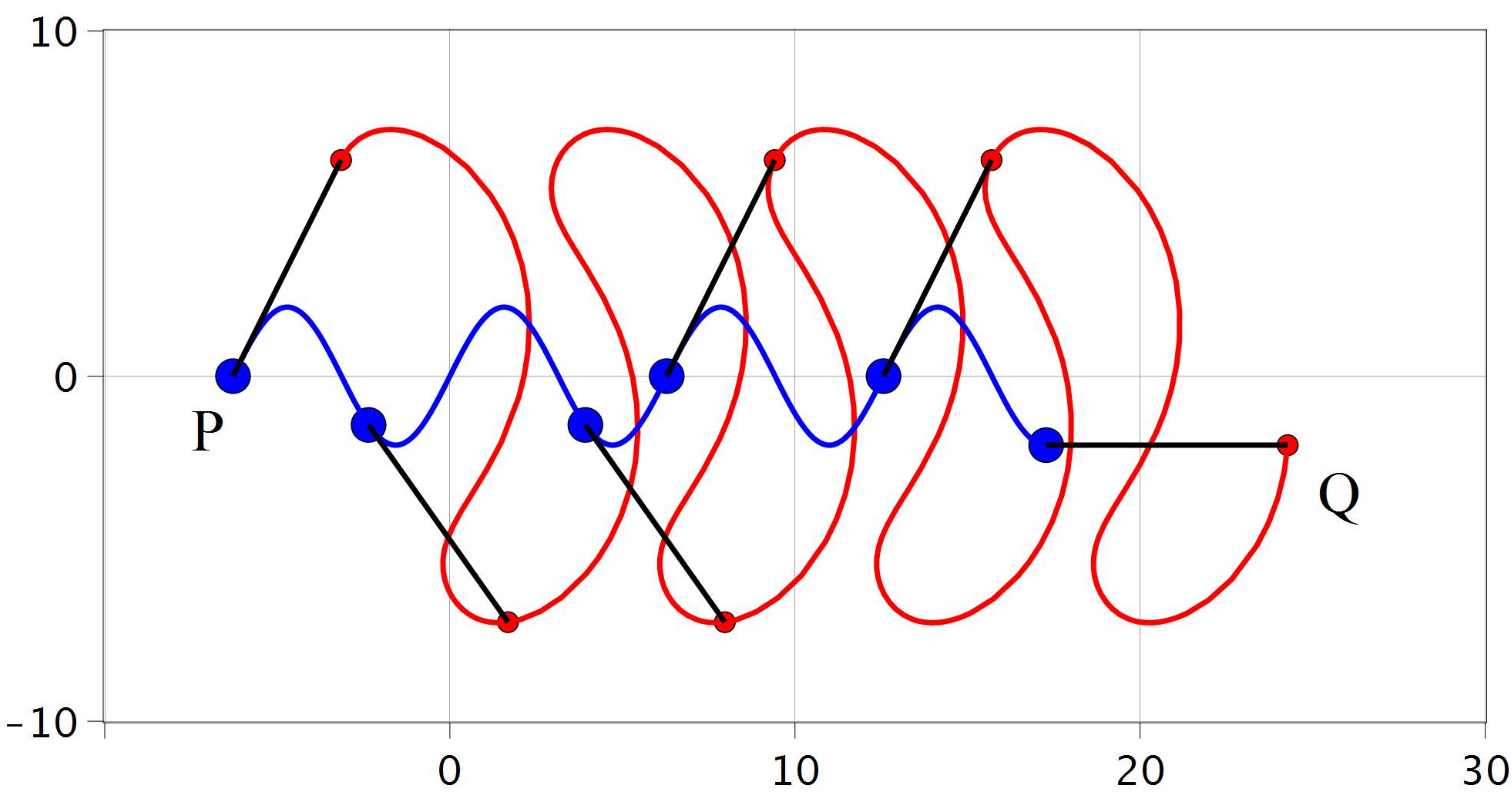}
\caption{Tractor (red) and  tractrix (blue). The pull of the tractor is from left to right. } \label{fig2DTractorContract1}
\end{center}
\end{figure}

\begin{figure}[h!]
\begin{center} \includegraphics[width=60mm]{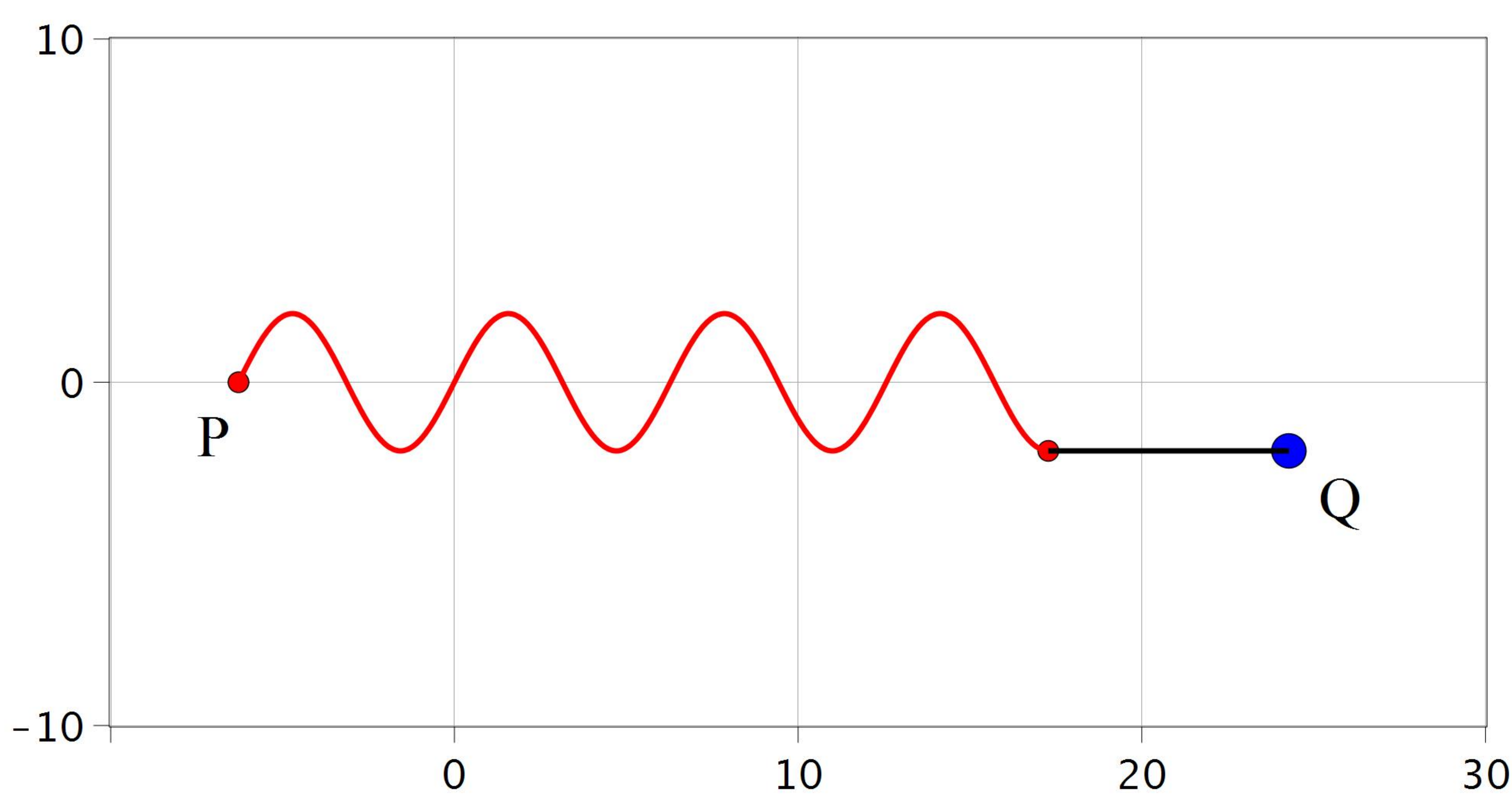} \includegraphics[width=60mm]{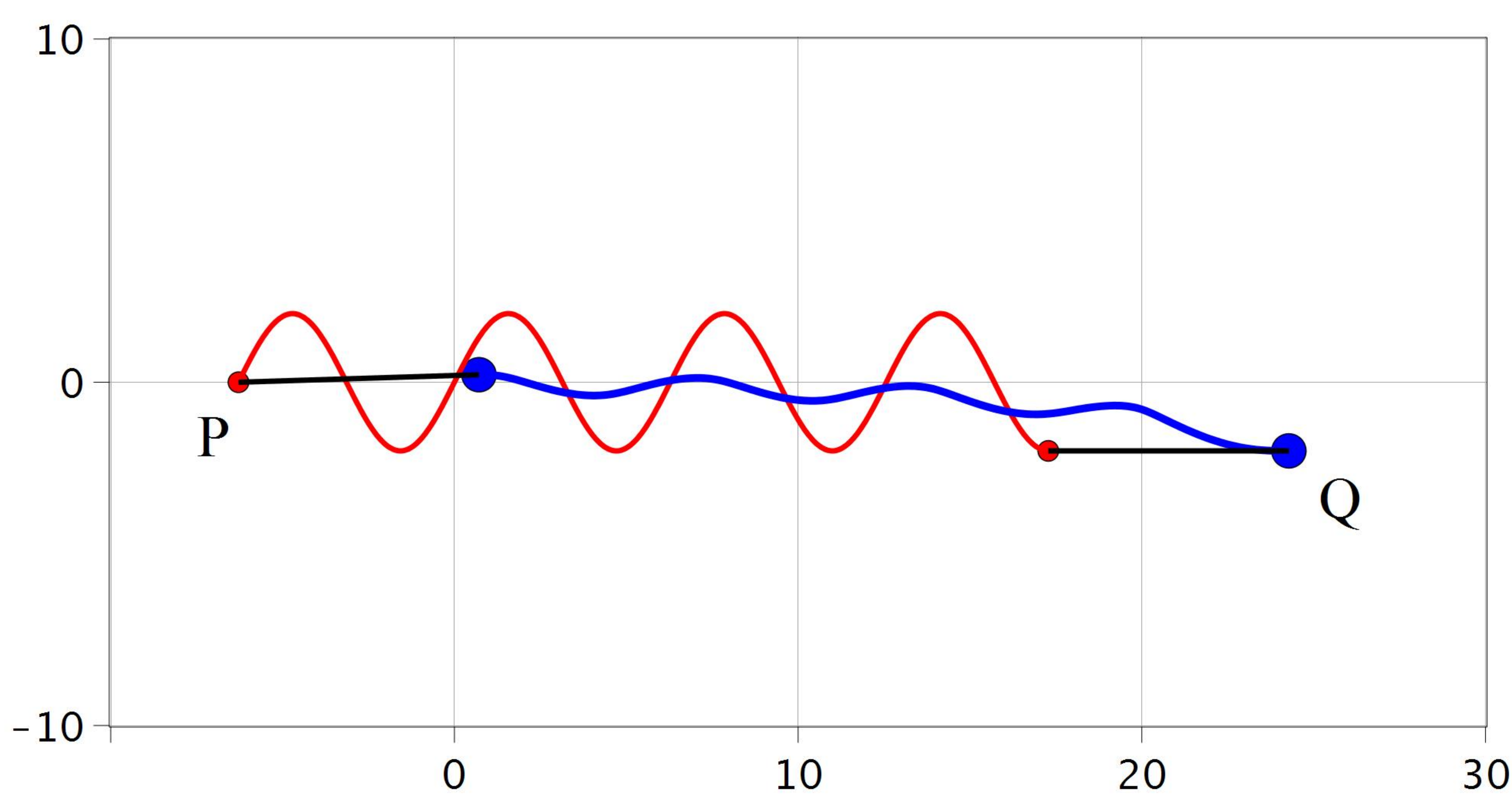}
\caption{Tractor (red, the tractrix curve from figure \ref{fig2DTractorContract1}) and tractrix (blue). The pull of the tractor is now from right to left. The length reduction from red tractor curve to blue tractrix curve is evident.} \label{fig2DTractorContract2}
\end{center}
\end{figure}

\begin{remark}
We note that if the tractrix curve $\gamma(s)$ is given, then the corresponding pulling and pushing tractor curves $\eta(t(s))$ are obtained directly as the endpoint curves for the geodesics of length $\ell$ issuing from $\gamma(s)$ in the directions of $\gamma^{\prime}(s)$ and $-\gamma^{\prime}(s)$, respectively. This is how figure \ref{fig2DTractorContract1} is constructed (using straight line geodesics in $\mathbb{R}^{2}$) and also how the figures \ref{figWildTrack1}--\ref{figWildTrack4} below are constructed (using a numerical procedure for developing the relevant geodesics on the given surface).
\end{remark}

A more general tractor/tractrix system -- which gives a first intuitive glimpse of the Riemannian systems under consideration  -- is displayed in the figures \ref{figWildTrack1}--\ref{figWildTrack4} which show a tractor track and a pulled tractrix on a surface in $\mathbb{R}^{3}$. Again the red tractor curve through the hilly region is much longer than the blue tractrix curve. Moreover, the area that is swept out by the wagon pole during the pull of the tractor is quite significant -- and obviously of great importance for the corresponding motion planning applications when lifted to surfaces as in the case shown here and when lifted into manifolds in general. We show in corollaries \ref{corSecUp} and \ref{corSecDown} how both the length of the tractor versus the length of the tractrix and the area of the pole-sweep can be estimated in terms of the total geodesic curvature of the tractrix together with the ambient sectional curvatures of the surface or of the manifold in question.

\begin{figure}[h!]
\begin{center}
\includegraphics[width=50mm]{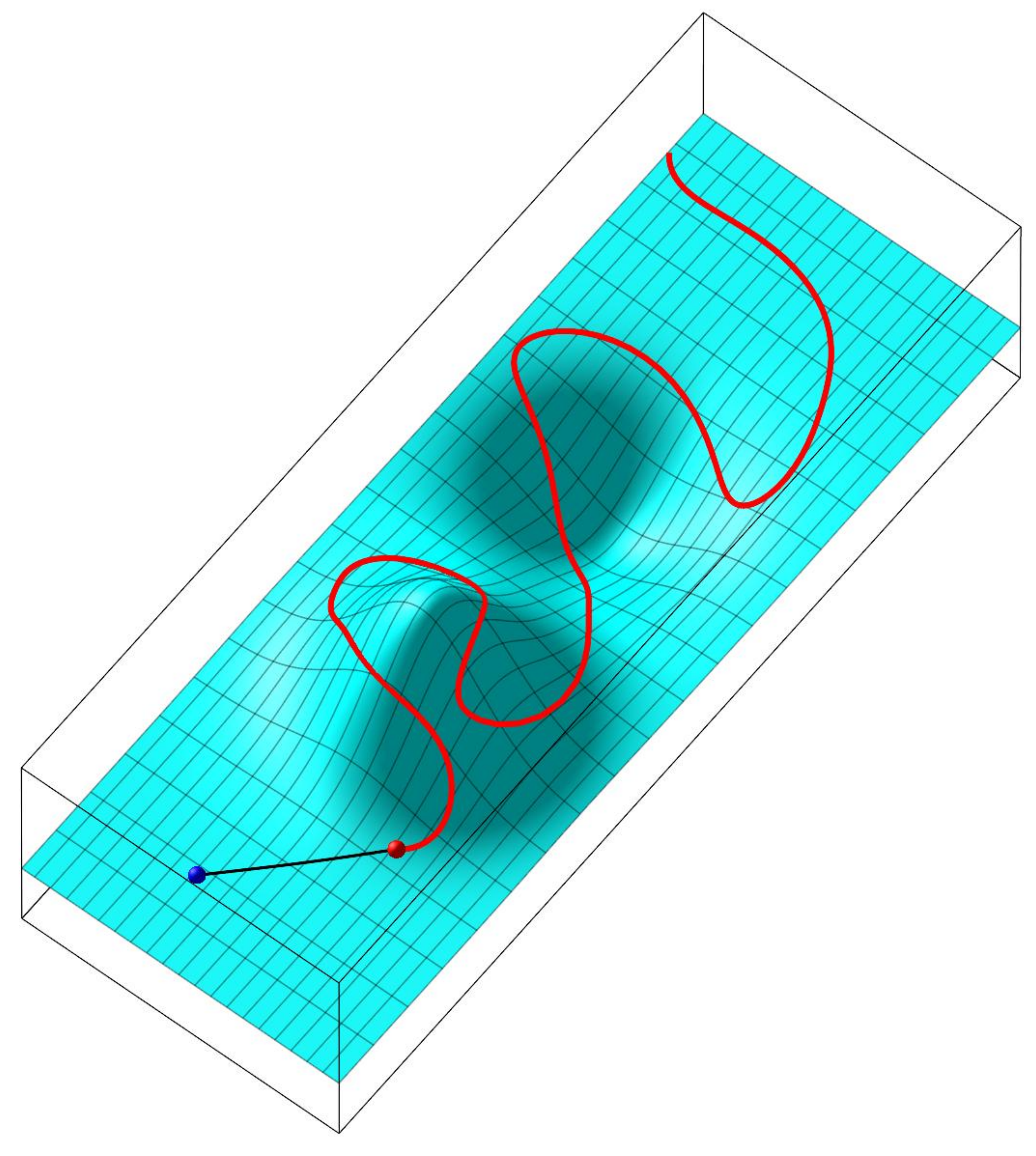}
\caption{A red tractor track on a hilly surface. The tractrix is initiated to the left using the black geodesic wagon pole.} \label{figWildTrack1}
\end{center}
\end{figure}

\begin{figure}[h!]
\begin{center}
\includegraphics[width=70mm]{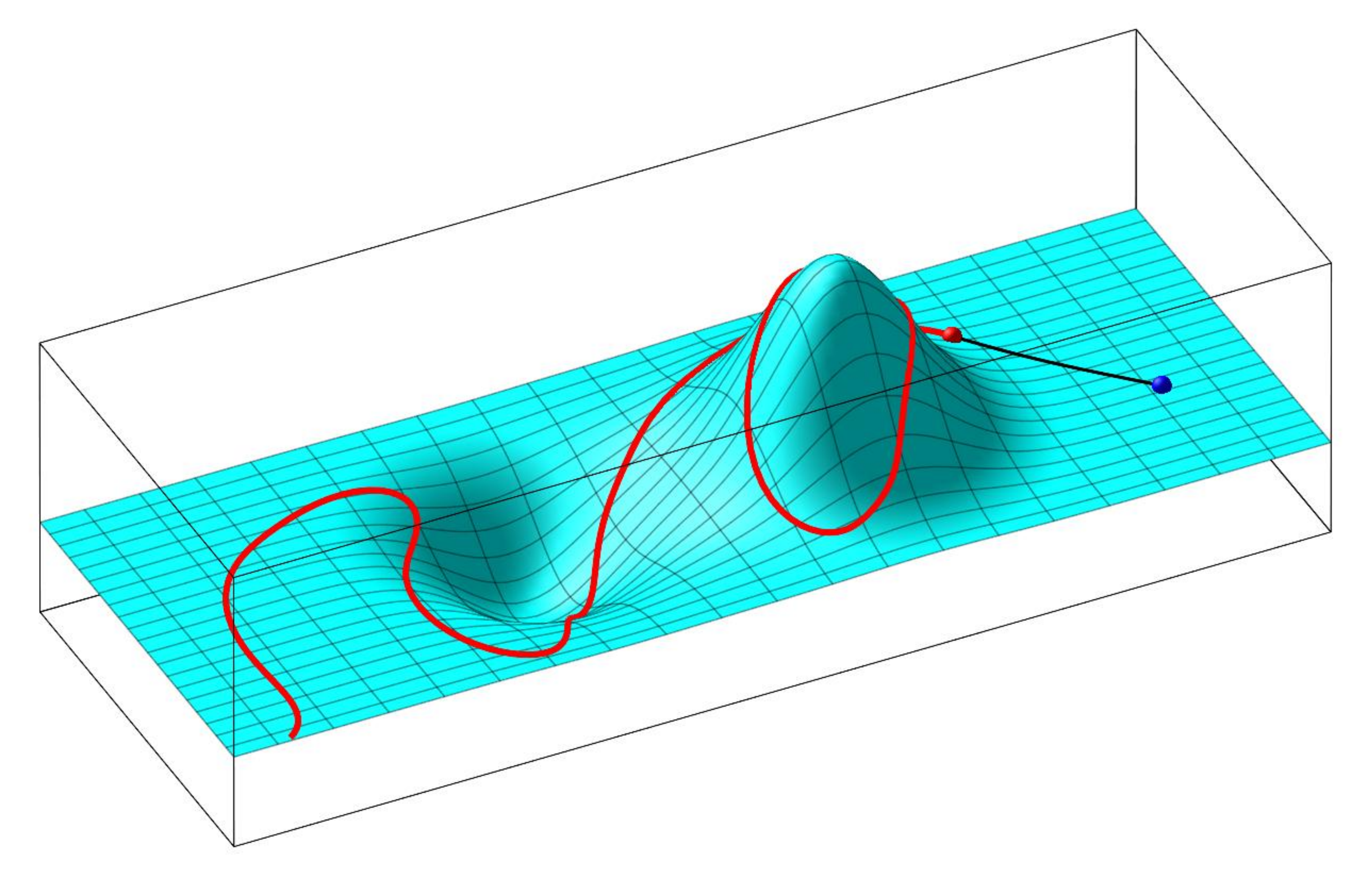}\includegraphics[width=70mm]{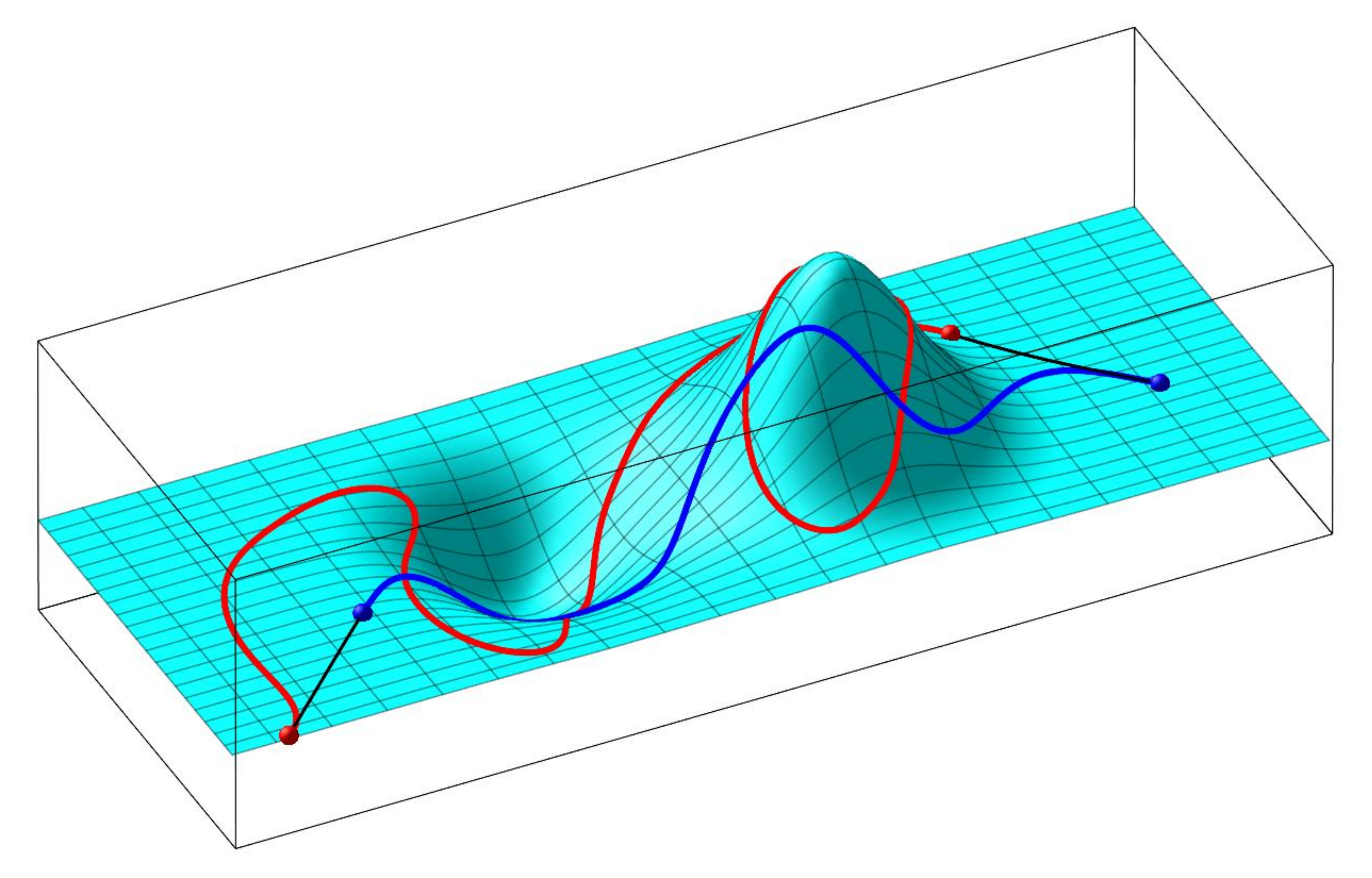}
\caption{The same situation as in figure \ref{figWildTrack1} -- from another viewpoint. To the right hand side is also shown the resulting blue tractrix curve. It is clearly shorter than the red tractor curve, cf. theorem \ref{thmMainLA} and corollary \ref{corLLdiffer}.} \label{figWildTrack2}
\end{center}
\end{figure}

\begin{figure}[h!]
\begin{center}
\includegraphics[width=60mm]{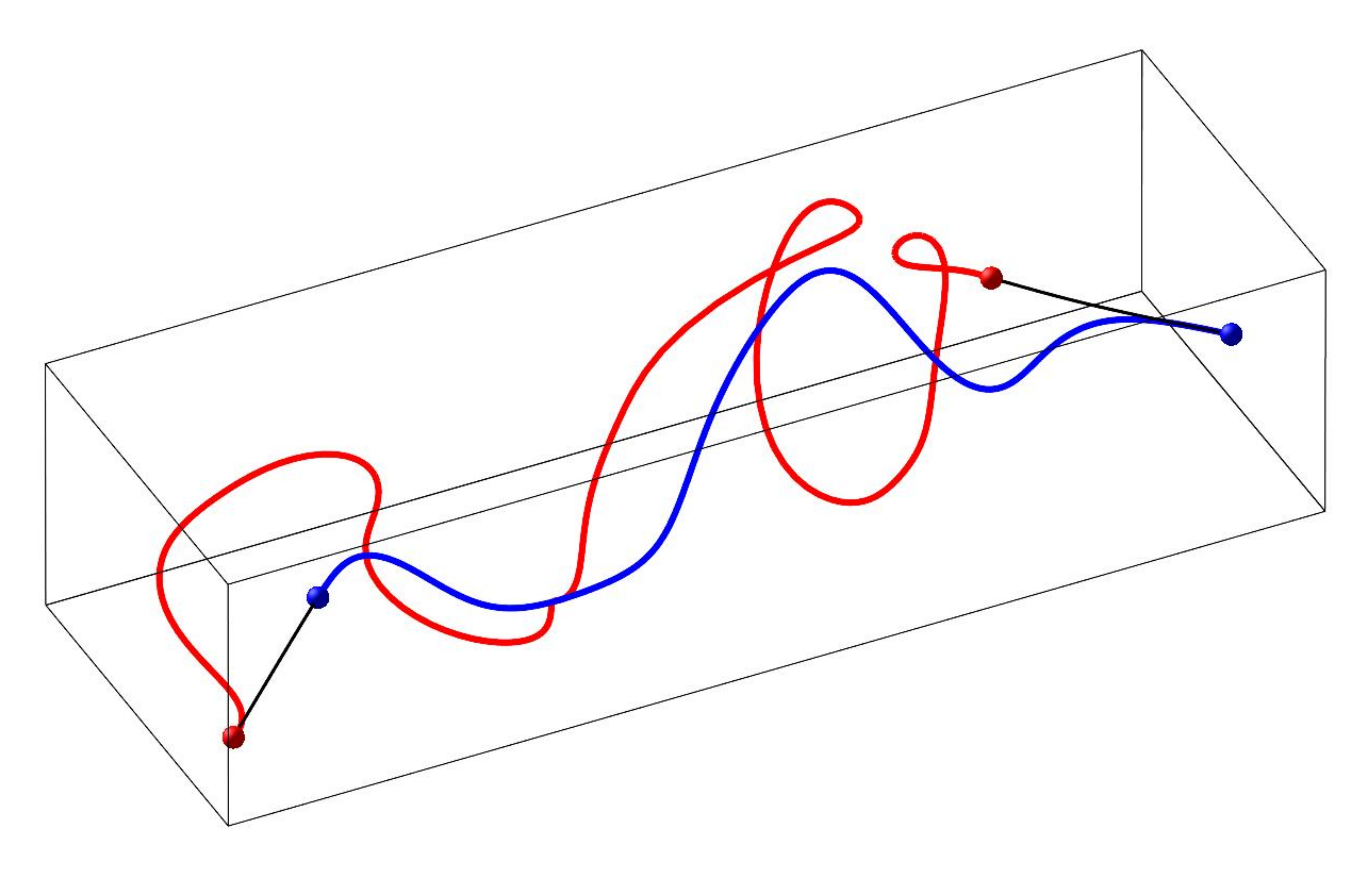}  \includegraphics[width=60mm]{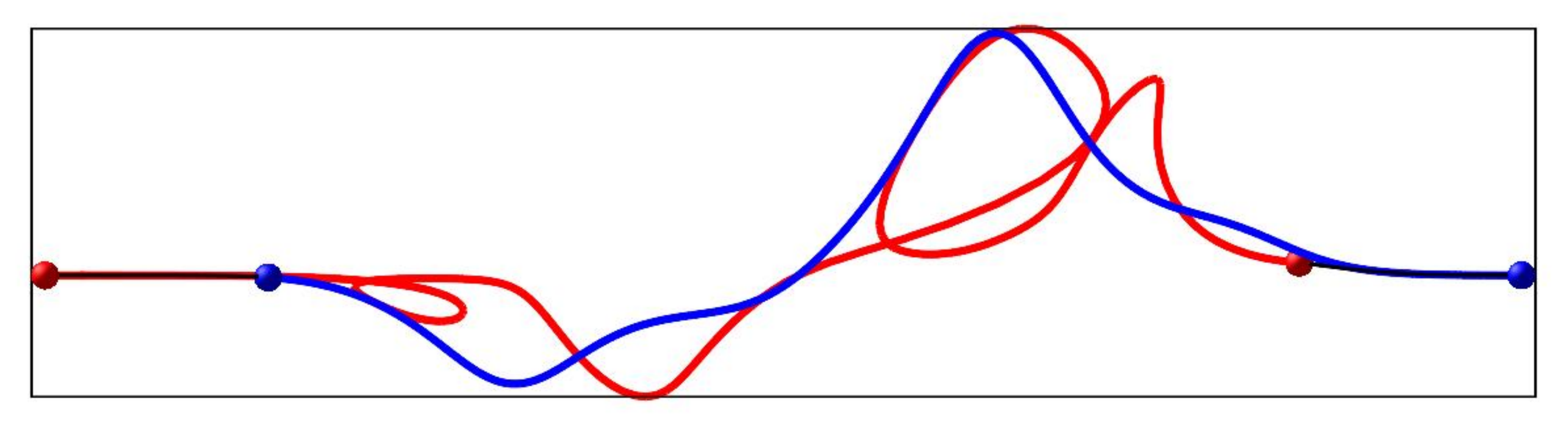} \quad
\caption{The tractor/tractrix curves from figure \ref{figWildTrack2}. To the right hand side is the display of the height profiles of the two curves.} \label{figWildTrack3}
\end{center}
\end{figure}

\begin{figure}[h!]
\begin{center}
\includegraphics[width=70mm]{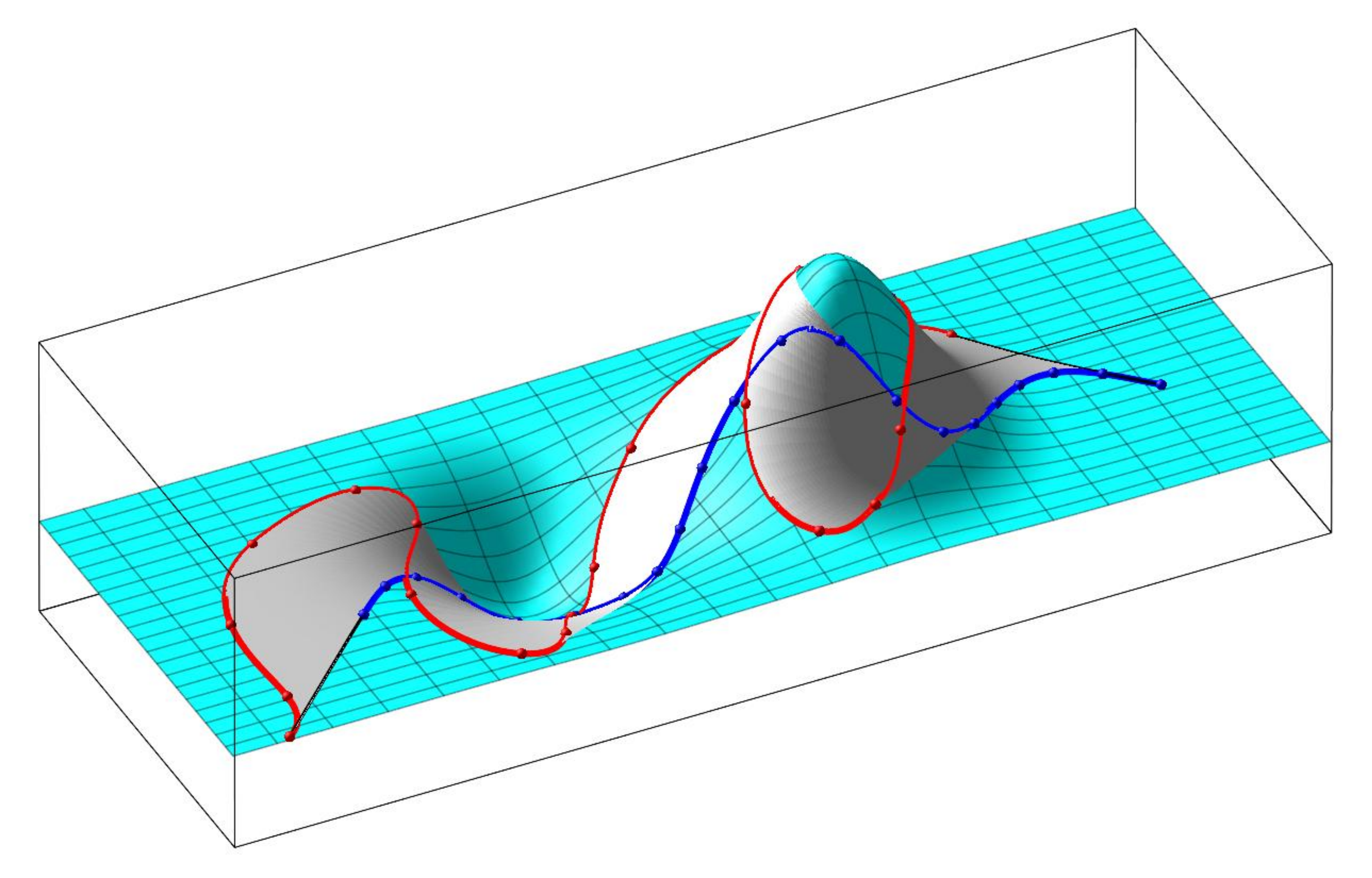}\includegraphics[width=70mm]{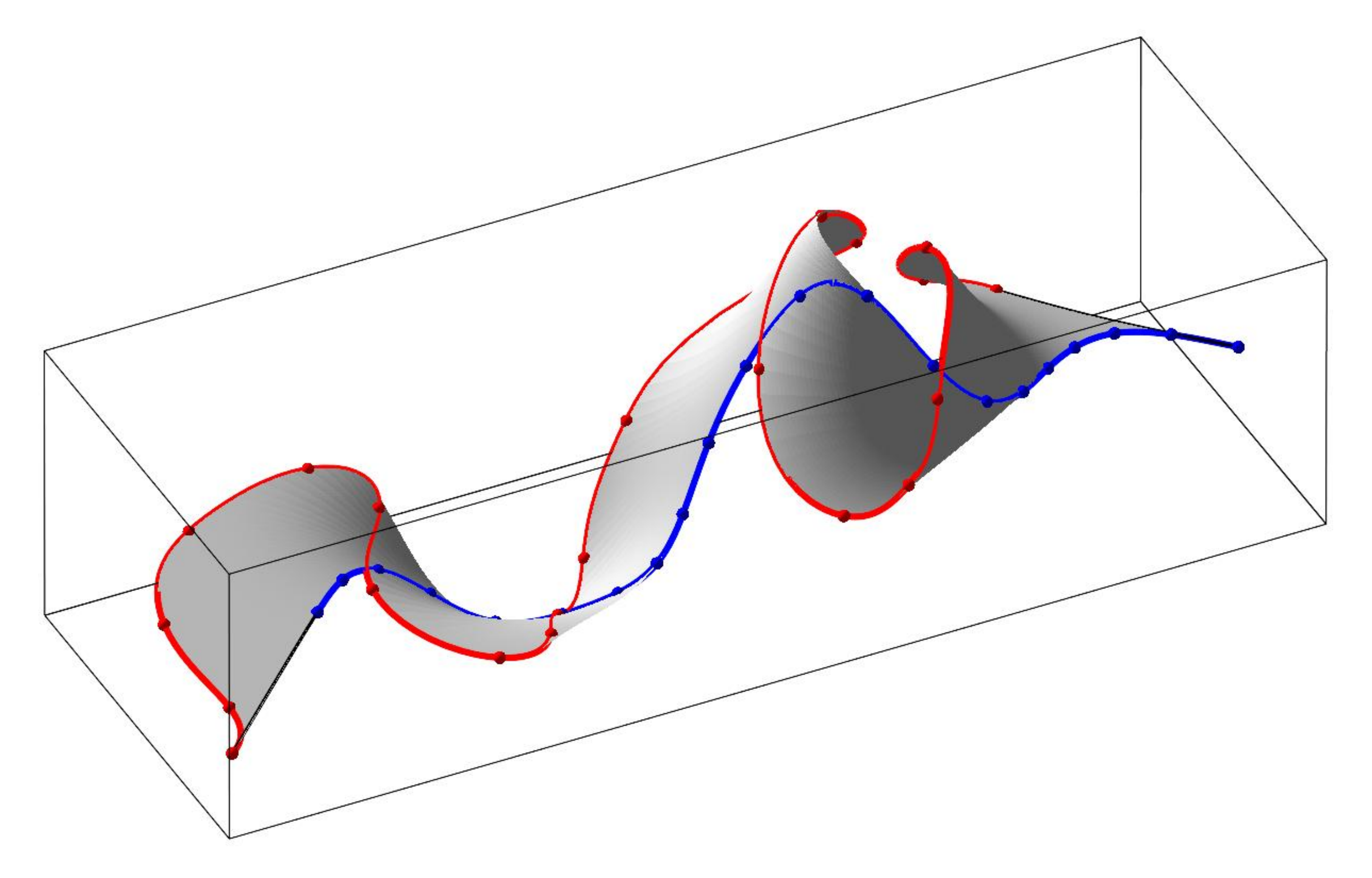}
\caption{The area swept out by the wagon pole during the pull of the tractor shown in figure \ref{figWildTrack2}, cf. theorem \ref{thmMainLA}.} \label{figWildTrack4}
\end{center}
\end{figure}


\section{The tractor/tractrix ODE system in $\mathbb{R}^{3}$} \label{sec3D}
The tractor/tractrix definition is easily interpreted in $\mathbb{R}^{3}$:

\begin{proposition}
In $\mathbb{R}^{3}$ we assume without lack of generality that the tractor/tractrix system is \emph{not} contained in a plane parallel to the $(x,y)$-plane at any time $t$.  Then the tractor/tractrix conditions are equivalent to the following ODE system which involves the given tractor curve $\eta(t) = (\eta_{1}(t), \eta_{2}(t),\eta_{3}(t))$, the ensuing tractrix $\gamma(t) = (\gamma_{1}(t), \gamma_{2}(t), \gamma_{3}(t))$ from a given starting point $\gamma(0) = p$, and the resulting pole $\lambda(t) = \eta(t) - \gamma(t)$:
\begin{equation}
\begin{aligned}
\lambda'(t)\cdot \lambda(t) &= 0\\
\left(\gamma'(t)\times \lambda(t)\right)\cdot(1,0,0) &= 0 \\
\left(\gamma'(t)\times \lambda(t)\right)\cdot(0,1,0) &= 0 \quad .
\end{aligned}
\end{equation}

\end{proposition}

\begin{proof}
The first equation is equivalent to  the condition, that the wagon pole length $\ell = \Vert \lambda(t) \Vert$ is constant. The other two equations imply that $\gamma^{\prime}(t)\times \lambda(t) = (0,0,h(t))$ for some function $h(t)$. If $h(t) \neq 0$ for some $t$, then  both  $\gamma^{\prime}(t)$ and $\lambda(t)$ are horizontal (parallel to the $(x,y)$-plane) in contradiction to our assumption. Thus $h(t) = 0$ and therefore $\gamma'(t)\times \lambda(t) = 0$ so that $\lambda(t)$ is parallel to $\gamma^{\prime}(t)$ whereever $\gamma^{\prime}(t) \neq 0$. Conversely, if $\lambda(t)$ is parallel to $\gamma^{\prime}(t)$ then the two equations are clearly satisfied.
\end{proof}

We illustrate a family of solutions based on helical tractors in figures  \ref{figAreaSweep1} and \ref{figAreaSweep2}.

\begin{figure}[h!]
\begin{center}
\includegraphics[width=40mm]{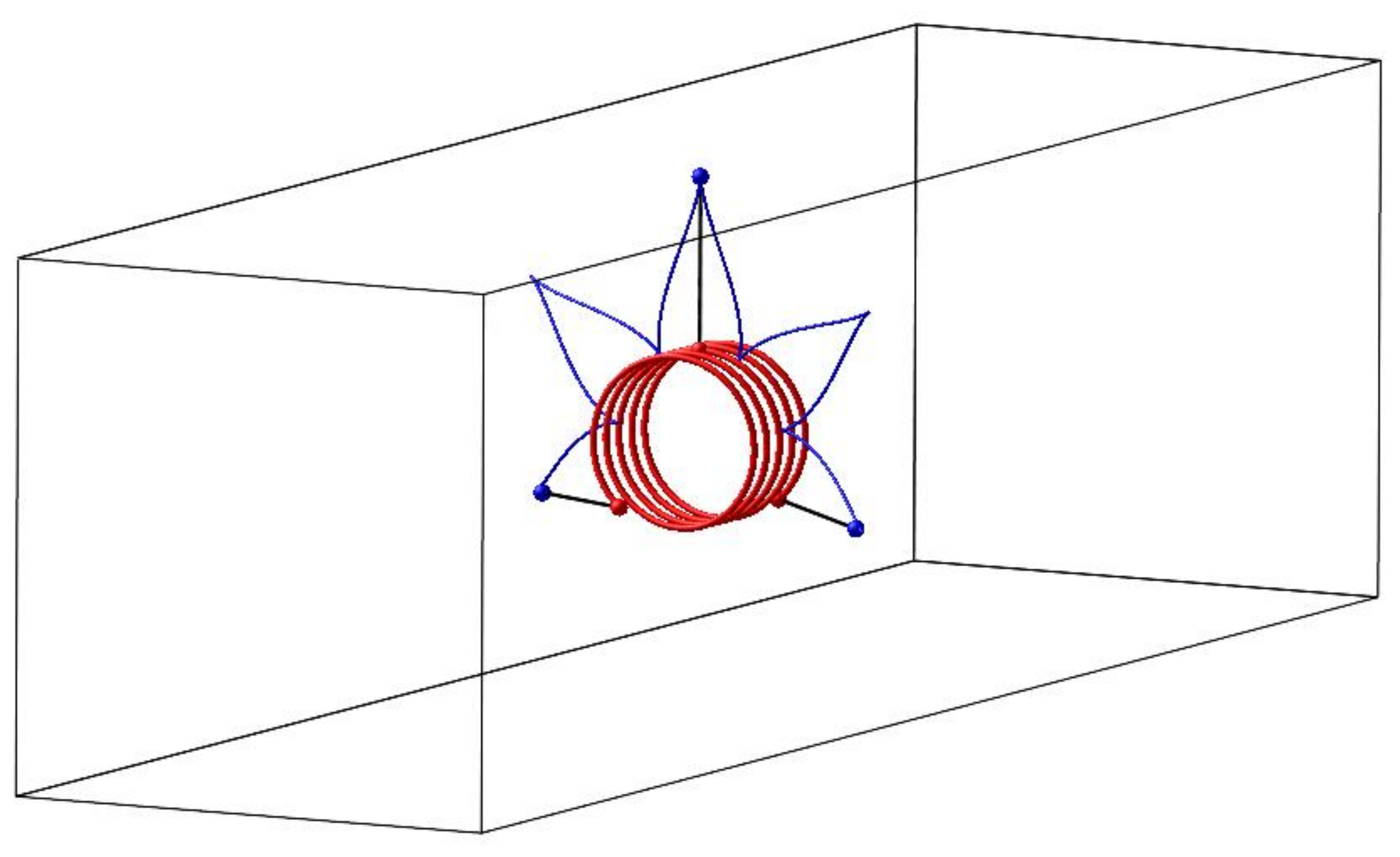}\includegraphics[width=40mm]{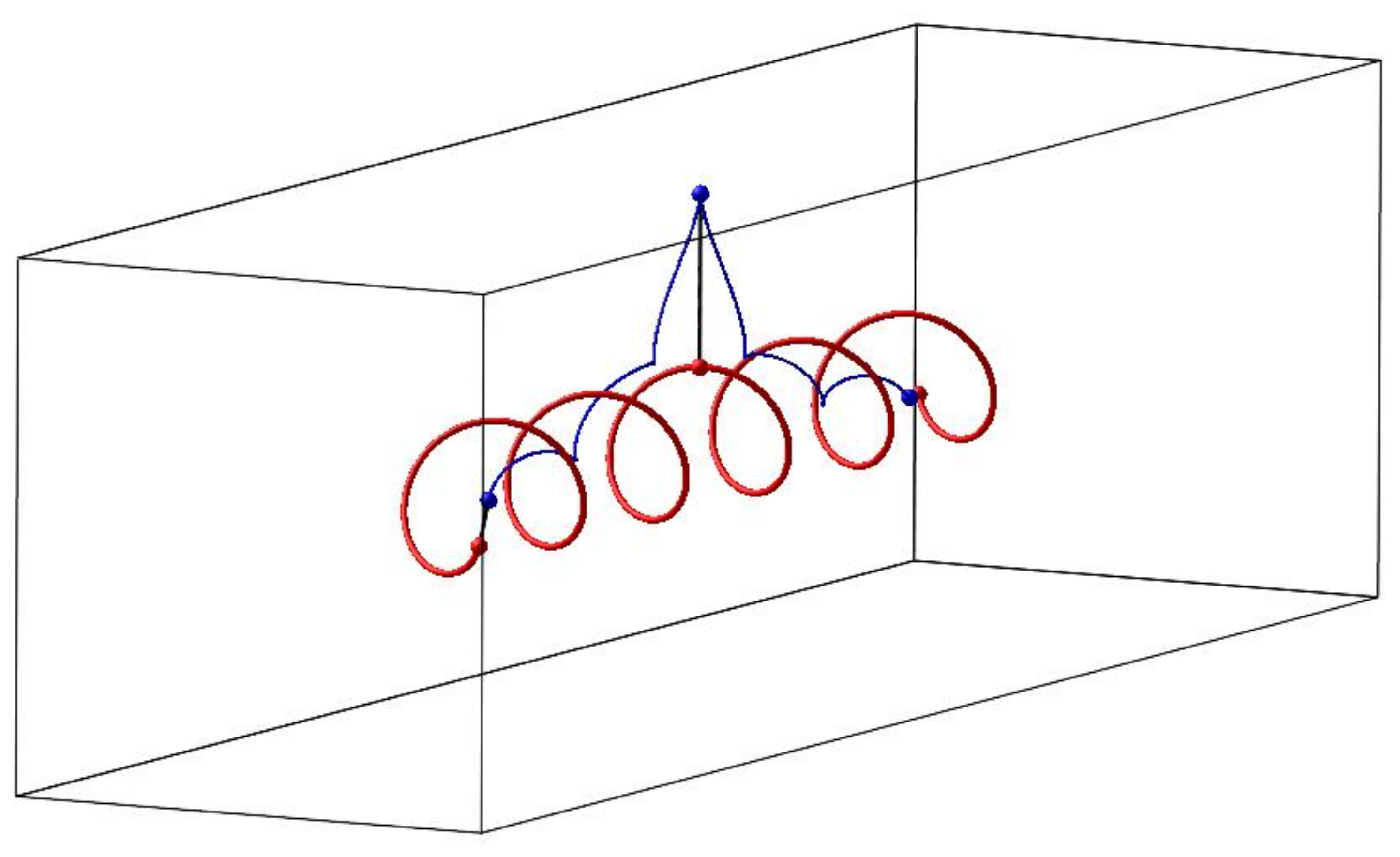}\includegraphics[width=40mm]{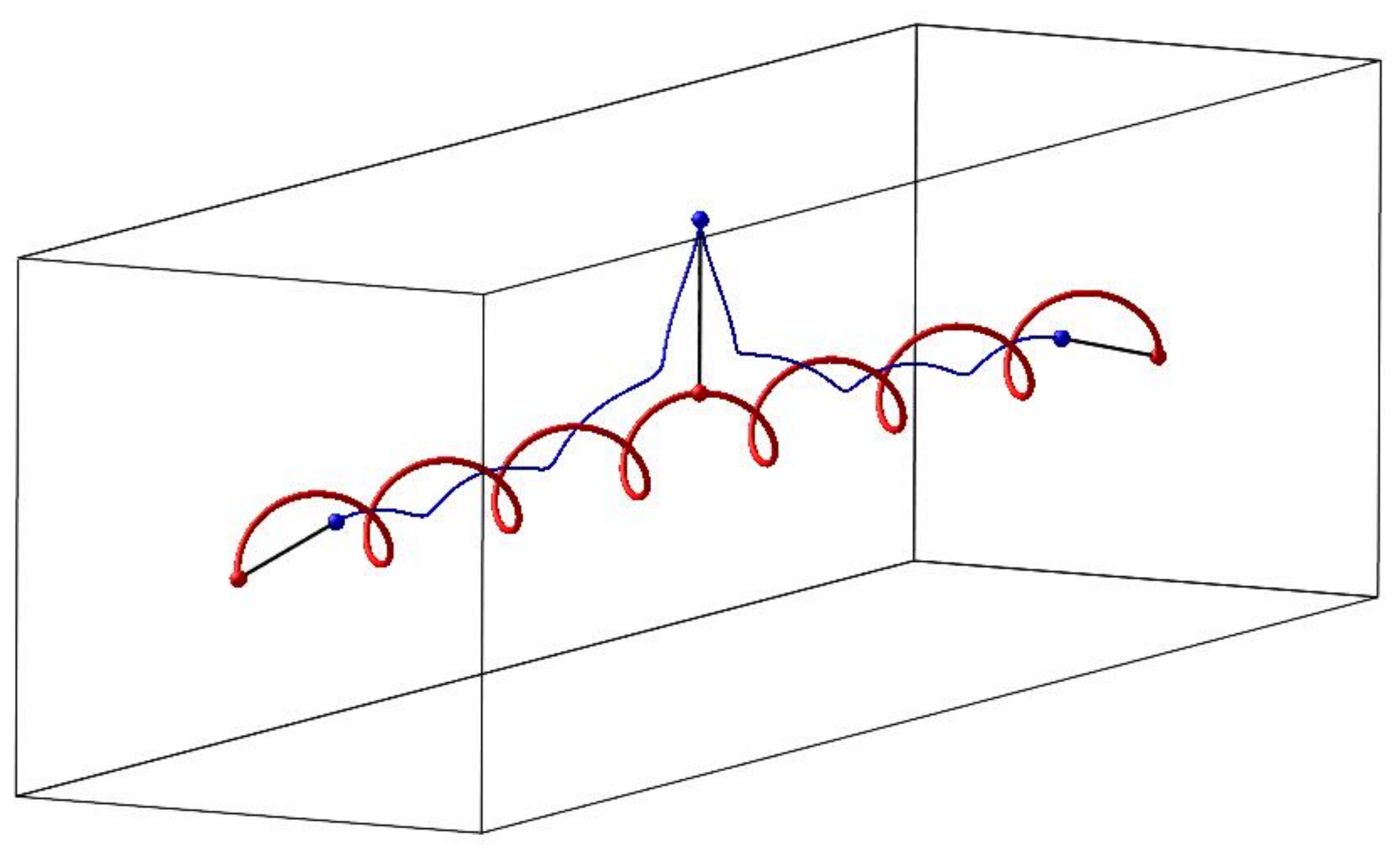}
\caption{Helical tractors and tractrices with one persistent cusp singularity in 3D. The tractor helices are chosen to have constant curvature and increasing torsion. } \label{figAreaSweep1}
\end{center}
\end{figure}

\begin{figure}[h!]
\begin{center}
\includegraphics[width=40mm]{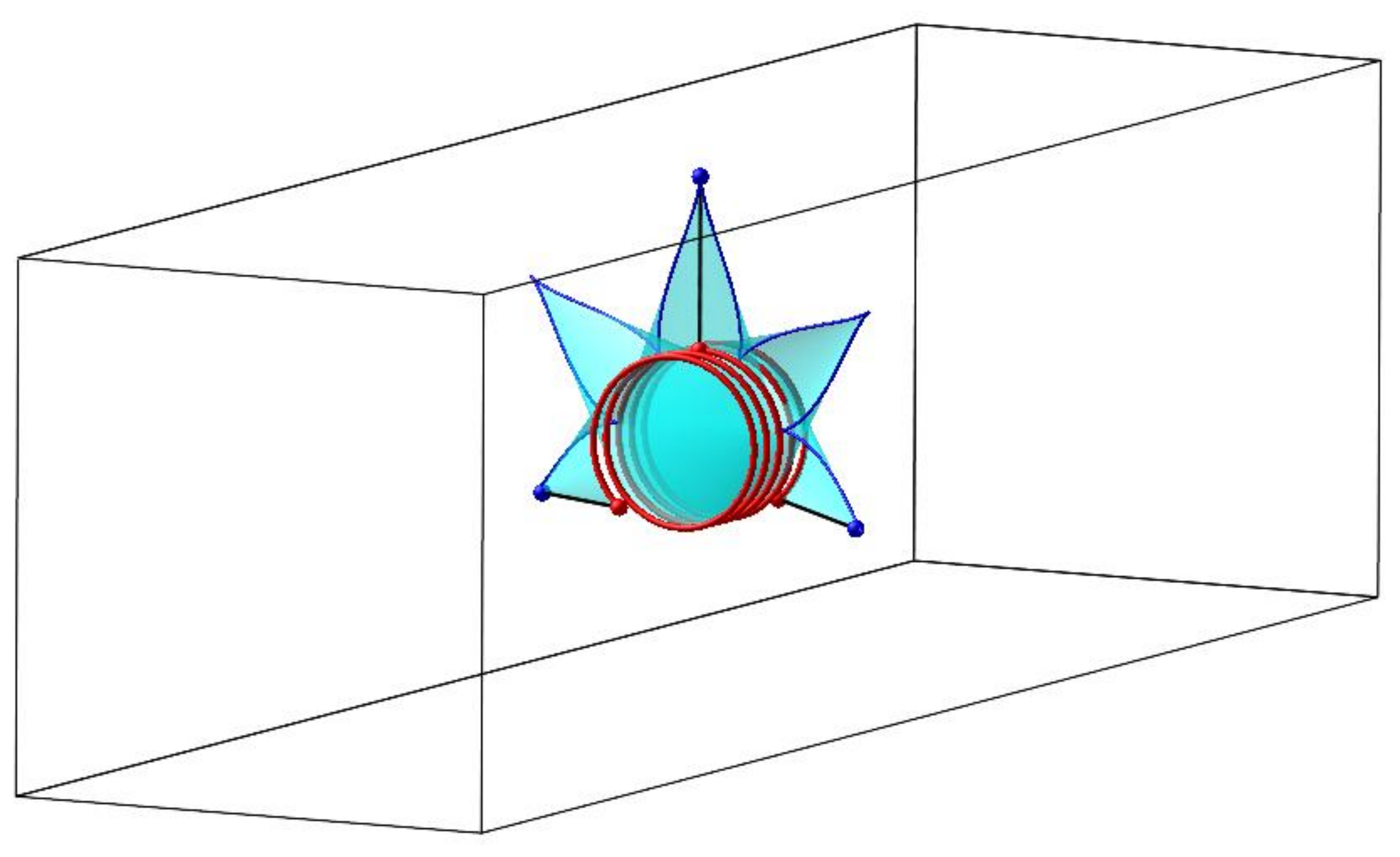}\includegraphics[width=40mm]{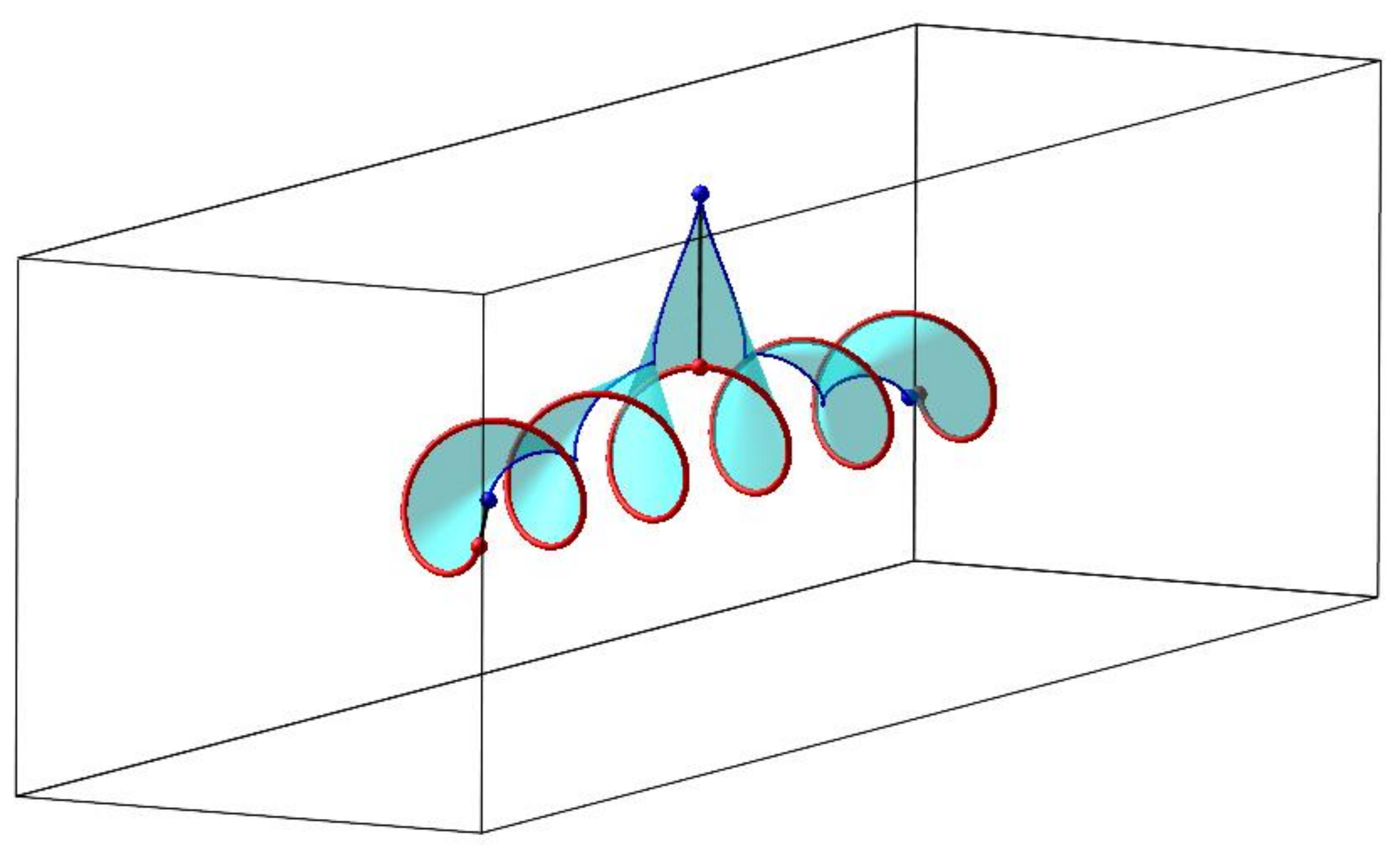}\includegraphics[width=40mm]{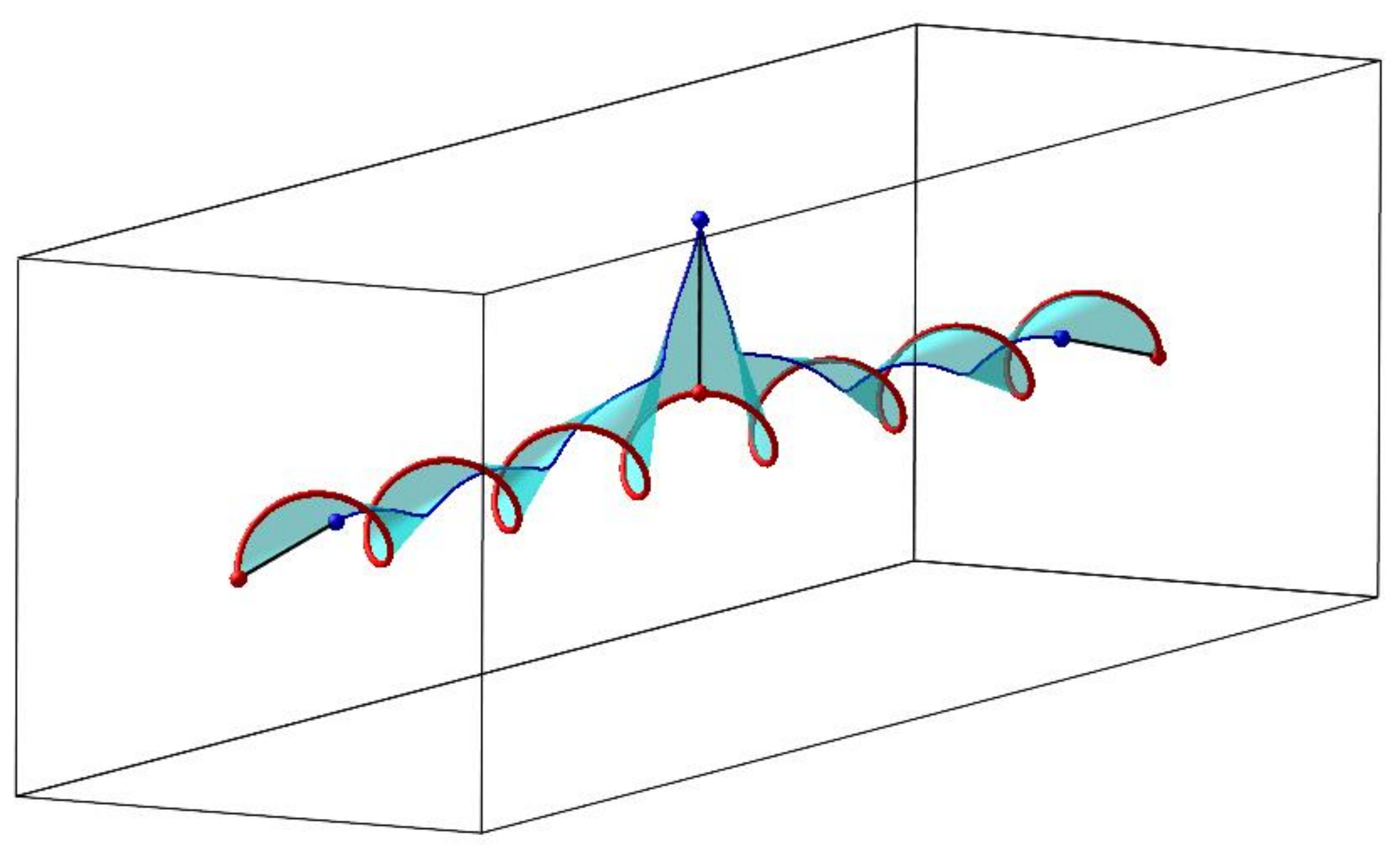}
\caption{Helical tractors, tractrices, and the corresponding wagon pole sweep surfaces in 3D. The surfaces which are swept out by the ruling wagon poles are flat tangent developable surfaces (with Gauss curvature $0$), as mentioned already in the introduction.} \label{figAreaSweep2}
\end{center}
\end{figure}


\section{The classical tractor/tractrix example in $\mathbb{R}^{2}$} \label{sec2D}

In $\mathbb{R}^{2}$, represented by the $(x,z)$-plane in $\mathbb{R}^{3}$, the tractor/tractrix conditions are equivalent to the following ODE system involving the given tractor $\eta(t) = (\eta_{1}(t), 0, \eta_{3}(t))$, the ensuing tractrix $\gamma(t) = (x(t), 0, z(t))$ from a given starting point $\gamma(0) = p$ at distance $\ell$ from $\eta(0)$, and the resulting pole $\lambda(t) = \eta(t) - \gamma(t)$:
\begin{equation} \label{eq2Dsystem}
\begin{aligned}
\lambda'(t)\cdot \lambda(t) &= 0\\
\left(\gamma'(t)\times \lambda(t)\right)\cdot(0,1,0) &= 0 \quad .
\end{aligned}
\end{equation}
Explicitly, in terms of the searched-for functions $x$ and $z$, these equations then read:
\begin{equation}
\begin{aligned}
(x' -\eta'_{1})\cdot(x - \eta_{1}) + (z' - \eta'_{3})\cdot(z - \eta_{3}) &= 0 \\
z'\cdot(x - \eta_{1}) + x'\cdot(z - \eta_{3}) &= 0 \quad .
\end{aligned}
\end{equation}

We briefly summarize some of the properties of the well-known classical tractrix, that will be generalized in the next sections:

The solution to the equations \eqref{eq2Dsystem} with the classical initial conditions $\gamma(0) = p = ( 0, 0, \ell)$ and $\eta(t) = (t, 0,0)$, $t\in [-T, T]$,   is displayed (using $\ell = 2$ and  $T=10$) in figure \ref{fig2DTractorClassic1} based on the exact expression:
\begin{equation}
\gamma(t) = \left( t -\ell\cdot \tanh\left(\frac{t}{\ell}\right)\, , \, \, 0 \, , \,\, \left(\frac{\ell}{\cosh\left(\frac{t}{\ell}\right)} \right)\right) \quad , \quad t \in [-T, T]\,.
\end{equation}
In the following we shall be mainly interested in the arc-length parametrizations of the tractrix curves. For the classical tractrix we find
\begin{equation}
\begin{aligned}
s(t) &= \int_{0}^{t} \Vert \gamma^{\prime}(u)\Vert \, du \\
&= \int_{0}^{t} \sign(u)\cdot\tanh(u/\ell)\, du \\
&= \sign(t)\cdot\ell\cdot \ln(\cosh(t/\ell)) \quad,
\end{aligned}
\end{equation}
so that
\begin{equation}
t(s) = \sign(s)\cdot\ell \cdot \arccosh\left( e^{\vert s \vert/\ell}\right)
\end{equation}
and thence by slight abuse of notation ($\gamma(t(s)) = \gamma(s)$) we get:
\begin{equation*}
\gamma(s) = \ell\cdot\left(\sign(s) \cdot \left(\arccosh(e^{\vert s \vert/\ell})-\sqrt{1-e^{-2\vert s \vert/\ell}}\right)\, , \, \, 0 \, , \, \, e^{-\vert s \vert/\ell} \right) \quad ,
\end{equation*}
where $s \in [-s(T), s(T)]$.

\begin{remark} \label{remPushPull}
We note that for $t<0$ (and thus $s<0$) the solution corresponds to a push of the wagon pole whereas for $t>0$ (and thus $s>0$) the solution describes a pull of the wagon pole. By reverting the direction of the tractor motion, the full tractrix curve can thus be obtained by using tractors that are only pulling or only pushing along the $x$-axis.
\end{remark}

For comparison with our main results below we calculate a few properties of this particular tractor/tractrix example. In view of the above remark \ref{remPushPull} we will focus on a 'positive' segment of the tractrix, i.e. $\gamma(s)$, $s \in [0, \LL(\gamma)]$.  We shall be particularly interested in the curvature $\kappa(s)$ of the tractrix, the orthogonal distance $\dist(s)$ from the tractrix to the tractor, the total curvature $\mathcal{K}$ of the tractrix, and the area $\A$ that is swept out by the wagon pole during the motion corresponding to  $s \in [0, \LL(\gamma)]$. It is straightforward to extract these values from the given parametrization:
\begin{equation} \label{eqConverg}
\begin{aligned}
\dist(s) &= \ell\cdot e^{-s/\ell} \quad , \\
\kappa(s) &= \Vert \gamma^{\prime\prime}(s)\Vert = \frac{e^{-s/\ell}}{\ell \sqrt{1-e^{-2s/\ell}}} \quad , \\
\mathcal{K} &= \arctan\left( \sqrt{e^{2\cdot \LL(\gamma)/\ell} -1}\right) \quad , \\
\A &= \frac{1}{2}\cdot \ell^2 \cdot \mathcal{K} \quad .
\end{aligned}
\end{equation}

\begin{remark}
The particular exponential decay to $0$ of both $\dist(s)$ and $\kappa(s)$ as $s\to \infty$, observed in \eqref{eqConverg}, is proved below to be a general phenomenon for geodesic tractors in spaces of bounded curvatures when the pole length $\ell$ is not too long -- see sections \ref{secSpaceFormPull} and \ref{secVariableToponogov}.
\end{remark}

\begin{remark}
We observe, that the curvature of the tractrix is $\kappa(0) = \infty$ at $s=0$, and that the total curvature is the total rotation of the wagon pole -- as it should be. In particular, this rotation is $\pi/2$ in the limit $\LL(\gamma) \to \infty$ corresponding to the total sweep area $\A = \ell^{2} \cdot \pi/4$.
\end{remark}

\begin{figure}[h!]
\begin{center}
\includegraphics[width=70mm]{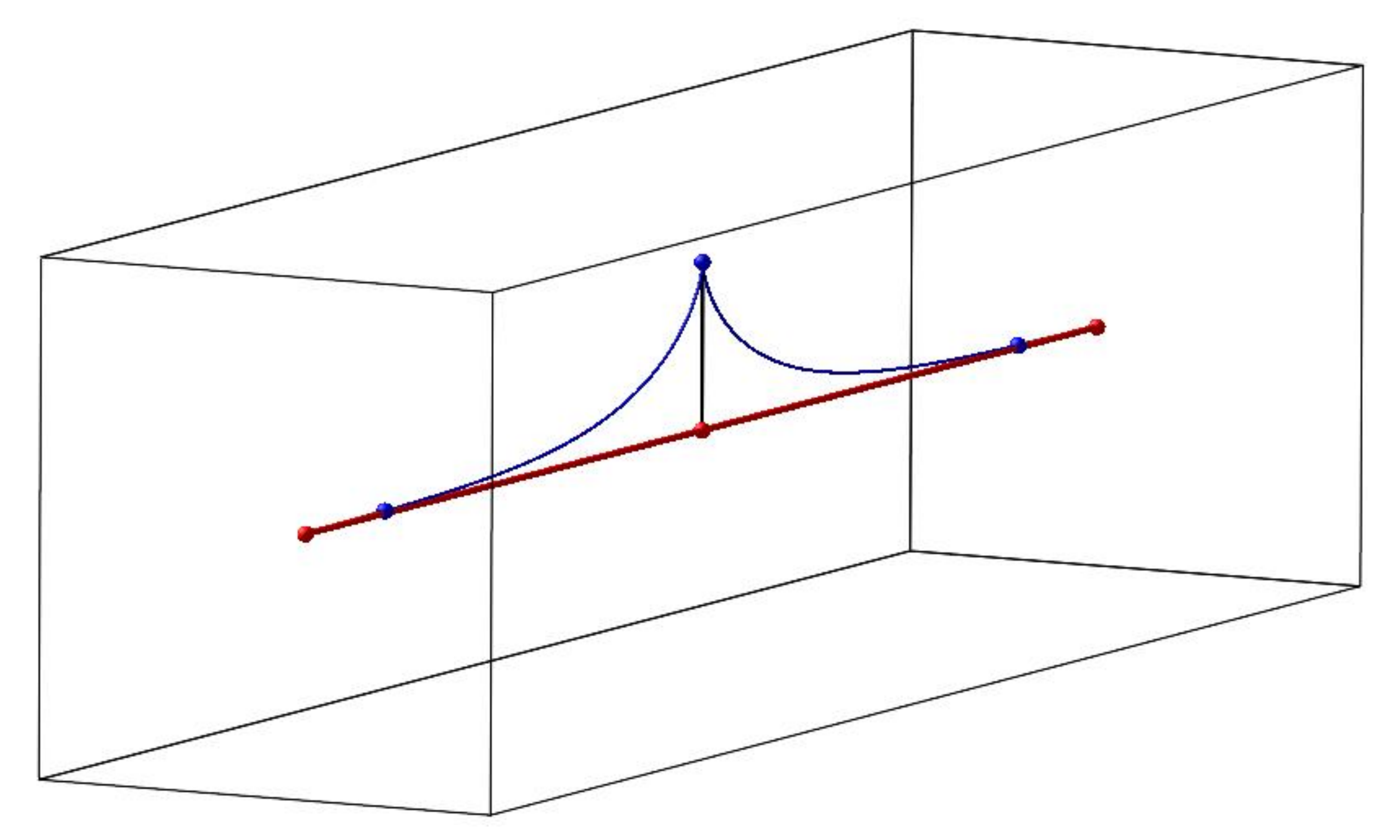}\includegraphics[width=70mm]{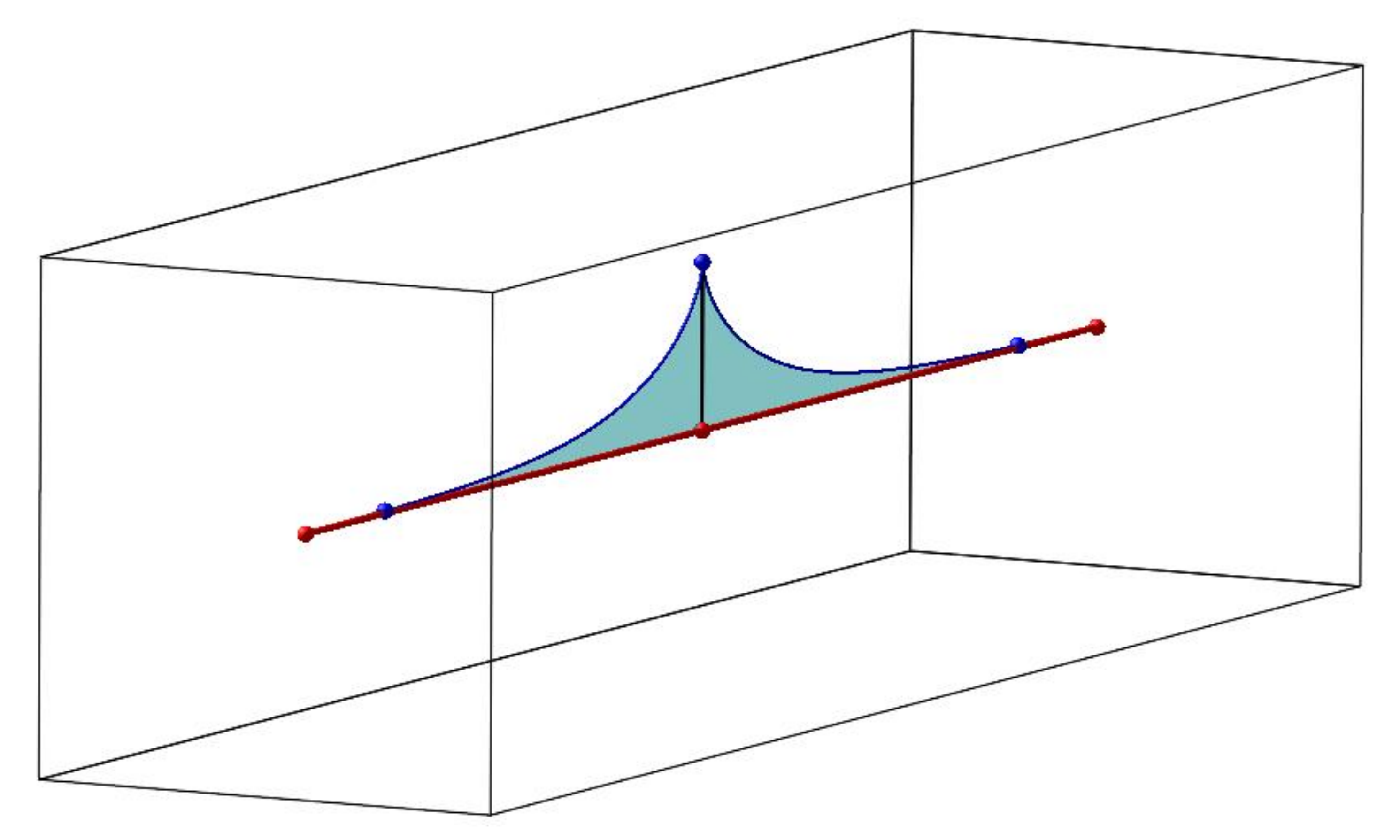}
\caption{The classical tractor/tractrix curve.} \label{fig2DTractorClassic1}
\end{center}
\end{figure}


\section{The wagon pole induced Jacobi fields\\and Rauch's theorems} \label{secJacobiField}

For each non-singular segment of the tractrix curve $\gamma(s)$, the wagon pole $\lambda_{s}(u)$ traces out a geodesic variation, which for each $s_{0}$ induces a Jacobi field $\Jv_{s_{0}}(u)$ along the specific wagon pole $\lambda_{s_{0}}(u)$ -- see \cite[Chpt. VIII]{KobNomII}:

\begin{proposition} \label{propJacobi}
The Jacobi field is split into the unit length tangential component $\lambda_{s_{0}}^{\prime}(u)$ and an orthogonal  component $\Jv_{s_{0}}(u)$.
The latter is determined by the geodesic curvature of the tractrix at $s_{0}$, defined by the covariant derivative as follows:
\begin{equation}
\Jv_{s_{0}}(0) = 0 \quad , \quad   \Jv_{s_{0}}^{\prime}(0) = D_{\gamma^{\prime}(s_{0})}\gamma^{\prime}(s_{0})
\end{equation}
together with the Jacobi equation, which introduces the influence of the curvature tensor $R$ of the ambient space $(M, g)$ along the wagon pole:
\begin{equation}
 D_{\lambda_{s_{0}}^{\prime}(u)}^{2}\Jv_{s_{0}}(u) + R\left(\Jv_{s_{0}}(u), \lambda_{s_{0}}^{\prime}(u)\right)\lambda_{s_{0}}^{\prime}(u) = 0 \quad.
\end{equation}
\end{proposition}

The geodesic curvature function for the tractrix curve is:
\begin{equation}
\kappa_{s_{0}} = \Vert D_{\gamma^{\prime}(s_{0})}\gamma^{\prime}(s_{0}) \Vert \quad ,
\end{equation}
and for $\kappa_{s_{0}} > 0$ we define the normalized length of the Jacobi field  $\Jv_{s_{0}}(u)$ by
\begin{equation}
J_{s_{0}}(u) = \Vert \Jv_{s_{0}}(u) \Vert/\kappa_{s_{0}} \quad.
\end{equation}
Then -- by linearity of the Jacobi equation -- the function $J_{s}(u)$ is directly comparable with the lengths of the standard space form Jacobi fields $\Jv^{K}(u)$ in ambient spaces of constant curvature $K$, see theorem \ref{thmRauchI} below.

We shall consider ambient manifolds $(M,g)$ with upper and/or lower bounds on their sectional curvatures $\sec_{M}$. Such conditions have well-known influences on the size of Jacobi fields $\Jv_{s}(u)$, see e.g. \cite[pp. 29--30]{CheegerEbin}, \cite[pp. 149]{Sakai}, \cite[p. 185]{Karcher1989}:

\begin{theorem}[Rauch I for upper curvature bound] \label{thmRauchI}
Let $\lambda(u)$ denote a unit speed geodesic (wagon pole) curve in a Riemannian manifold $(M, g)$ with sectional curvatures bounded from above by a constant $K$, i.e. $\sec_{M} < K$. If $K > 0$ we assume further that $\mathcal{L}(\lambda) = \ell < \pi/\sqrt{K}$. Then
\begin{equation}
J_{s}(u) \leq J^{K}(u) \quad \textrm{for all} \quad u \in [0, \ell] \quad ,
\end{equation}
where $J^{K}(u)$ denotes the norm of the Jacobi fields in constant curvature $K$:
\begin{equation} \label{eqJK}
J^{K}(u) = \left\{
             \begin{array}{ll}
               \frac{1}{k}\cdot \sin(k\cdot u) & \hbox{for $K = k^{2} > 0$} \\
               u & \hbox{for $K=0$} \\
               \frac{1}{k}\cdot \sinh(k\cdot u) & \hbox{for $K= -k^{2} < 0$}\quad .
             \end{array}
           \right.
\end{equation}
\end{theorem}

\begin{theorem}[Rauch II for lower curvature bound] \label{thmRauchII}
Let $\lambda(u)$ denote a unit speed geodesic (wagon pole) curve in a Riemannian manifold $(M, g)$ with sectional curvatures bounded from below by a constant $K$, i.e. $\sec_{M} > K$ and assume further that $J_{s}(u) > 0$ for all $0< s < \ell$, i.e. $\lambda$ does not reach its first conjugate point from $\lambda(0)$. Then
\begin{equation}
J_{s}(u) \geq J^{K}(u) \quad \textrm{for all} \quad u \in [0, \ell] \quad .
\end{equation}
\end{theorem}


\section{Tractor/tractrix consequences\\from Rauch's theorems } \label{secMainLA}

\begin{theorem}\label{thmMainLA}
In consequence of proposition \ref{propJacobi}, we obtain the following expressions for the length $\mathcal{L}(\eta)$ of the tractor curve and for the area $\mathcal{A}$ swept out by the wagon pole:
\begin{equation}
\mathcal{L}(\eta) = \int_{0}^{\LL(\gamma)}\sqrt{1 + \kappa^{2}(s)\cdot J_{s}^{2}(\ell)} \, ds\quad ,
\end{equation}
and
\begin{equation}
\mathcal{A} = \int_{0}^{\LL(\gamma)}\int_{0}^{\ell}\kappa(s)\cdot J_{s}(u)\, du \, ds \quad .
\end{equation}
\end{theorem}

\begin{proof}
Combining the (unit length) longitudinal part and the orthogonal part of the Jacobi field $\Jv_{s_{0}}(u)$ we get immediately from $\eta^{\prime}(s) = \Jv_{s}(\ell)$:
\begin{equation}
\Vert \eta^{\prime}(s_{0}) \Vert = \sqrt{1 + \kappa^{2}(s_{0})\cdot J_{s_{0}}^{2}(\ell)} \quad ,
\end{equation}
from which the length of $\eta$ follows:
\begin{equation} \label{eqGeneralL}
\LL(\eta) = \int_{0}^{\LL(\gamma)} \sqrt{1 + \kappa^{2}(s)\cdot J_{s}^{2}(\ell)}\, ds \quad .
\end{equation}

The area swept out by the wagon pole during the motion is similarly determined by the ensuing
parametrization of the sweep:
\begin{equation}
r(s, u) = \lambda_{s}(u) \quad , \quad s \in [0, \mathcal{L}(\gamma)]\, , \, \,  u \in [0, \ell] \quad,
\end{equation}
with the Jacobian determinant function -- using the notation $\lambda^{\prime}_{s}(u)$ for the $u$-derivative of $\lambda_{s}(u)$:
\begin{equation}
\Jac_{r}(s,u) = \A(\spanning(\lambda^{\prime}_{s}(u), \mathcal{J}_{s}(u))) = \kappa(s)\cdot J_{s}(u)\quad ,
\end{equation}
so that
\begin{equation}
\mathcal{A} = \int_{0}^{\LL(\gamma)}\int_{0}^{\ell}\kappa(s)\cdot J_{s}(u)\, du \, ds \quad .
\end{equation}
\end{proof}

In order to estimate explicitly the length difference $\LL(\eta) - \LL(\gamma)$ we
note that \eqref{eqGeneralL} implies:
\begin{corollary}\label{corLLdiffer}
With the notation as above:
\begin{equation} \label{eqLLdiffer}
\LL(\eta) - \LL(\gamma) \geq \LL(\gamma)\cdot \left( \sqrt{1 + \left(\frac{1}{\LL(\gamma)}\cdot \int_{0}^{\LL(\gamma)}\kappa(s)\cdot J_{s}(\ell)\right)^{2}} -1\right) \, ds\quad.
\end{equation}
\end{corollary}
\begin{proof}
We use the following inequality, which is obtained by estimating the length of the graph of the function $\int f(x)\, dx$, $x \in [0, a]$:
\begin{equation}
\int_{0}^{a} \sqrt{1+ f^{2}(u)} \, du \geq \sqrt{a^{2} + \left(\int_{0}^{a} f(u) \, du  \right)^2} \quad.
\end{equation}
\end{proof}

The following observation is very useful for applying tractor/tractrix systems to curve shortening processes, since it says that the shortening only stops when a geodesic is reached -- see corollary \ref{corLLdiffer} above:

\begin{proposition} \label{propEqualLength}
The length of the tractor and the length of the tractrix are equal if and only if they are both geodesic curves.
\end{proposition}
\begin{proof}
For equality in equation \eqref{eqLLdiffer}, either $\kappa(s) = 0$ in which case the pole is just a geodesic prolongation of the tractrix curve, resulting in a geodesic tractor curve, or $J_{s}(\ell) =0$ in which case the pole is again a geodesic prolongation of the tractrix curve with the same consequence.
\end{proof}


\subsection{Ambient curvature dependence}
In view of the Rauch comparison theorems we get the following corollaries, which generalize the corresponding results for tractor/tractrix systems in $\mathbb{R}^{2}$.

\begin{corollary}
Assume that  the ambient space $(M, g)$ has constant curvature $K$. Then every tractor/tractrix system (assuming $k \cdot \ell < \pi$ if $K = k^{2} > 0$) gives the following relation between the total geodesic curvature $\mathcal{K} = \int_{0}^{\LL(\gamma)}\kappa(s) \, ds$ of $\gamma$, the length of $\eta$, and the area of the $\lambda$-sweep (including multiple coverings):
\begin{equation}
\begin{aligned}
\LL(\eta) &= \int_{0}^{\LL(\gamma)} \sqrt{1 + \kappa^{2}(s)\cdot (J^{K}(\ell))^{2}}\, ds \\
& \geq \sqrt{\LL^{2}(\gamma) + \left(J^{K}(\ell) \cdot \mathcal{K}(\gamma)\right)^{2}} \\
& \geq \LL(\gamma) \quad .
\end{aligned}
\end{equation}
with equalities if and only if the tractrix (and hence also the tractor) is a geodesic curve in $(M, g)$ -- in accordance with the general observation in proposition \ref{propEqualLength}.
Moreover, for the sweep-area we get the explicit formulas:
\begin{equation} \label{eqAreaConstK}
\begin{aligned}
\mathcal{A} &= \int_{0}^{\LL(\gamma)}\int_{0}^{\ell}\kappa(s)\cdot J^{K}(u)\, du \, ds \\
&=  \int_{0}^{\ell}J^{K}(u)\, du \cdot \int_{0}^{\LL(\gamma)}\kappa(s) \, ds \\ \\
& = \left\{
             \begin{array}{ll}
               \left(\frac{1}{k^{2}}\right)\cdot (1 - \cos(k\cdot \ell))\cdot \mathcal{K}(\gamma)   & \hbox{for $K=k^{2} >0$} \\ \\
               \left(\frac{1}{2}\right)\cdot \ell^{2}\cdot \mathcal{K}(\gamma)  & \hbox{for $K=0$} \\ \\
                \left(\frac{1}{k^{2}}\right)\cdot (\cosh(k\cdot \ell)-1)\cdot \mathcal{K}(\gamma) & \hbox{for $K=-k^{2} <0$.}
             \end{array}
           \right. \\
\end{aligned}
\end{equation}
\end{corollary}

The corresponding comparison results (with just an upper or a lower bound on the ambient curvatures) now follow directly from the Rauch theorems:

\begin{corollary} \label{corSecUp}
Assume that  the ambient space $(M, g)$ has sectional curvatures bounded from above, i.e. $\sec_{M} < K$. Then
\begin{equation}
\begin{aligned}
\LL(\eta) &\geq \int_{0}^{\LL(\gamma)}\sqrt{1 + \left(\kappa(s)\cdot J^{K}(\ell)\right)^{2}}\, ds\\
&\geq \sqrt{\LL^{2}(\gamma) + (J^{K}(\ell)\cdot \mathcal{K})^{2}} \\
&\geq \LL(\gamma) \quad ,
\end{aligned}
\end{equation}
and
\begin{equation}
\begin{aligned}
\mathcal{A} &\geq \mathcal{K}(\gamma)\cdot \int_{0}^{\ell}J^{K}(u) \,du \quad,
\end{aligned}
\end{equation}
where the last integral is as evaluated explicitly in \eqref{eqAreaConstK}.
Equalities occur if and only if $\gamma$ and $\eta$ are geodesics.
\end{corollary}

\begin{corollary} \label{corSecDown}
Assume that  the ambient space $(M, g)$ has sectional curvatures bounded from below, i.e. $\sec_{M} > K$. Then
\begin{equation}
\begin{aligned}
\LL(\eta) &\leq \int_{0}^{\LL(\gamma)}\sqrt{1 + (\kappa(s)\cdot J^{K}(\ell))^{2}} \, ds\\
&\leq \LL(\gamma) + J^{K}(\ell) \cdot \mathcal{K}(\gamma) \quad ,
\end{aligned}
\end{equation}
and equalities occur if and only if $\gamma$ and thence $\eta$ are geodesic curves. For the area of the pole-sweep we get similarly:
\begin{equation}
\begin{aligned}
\mathcal{A} \leq \mathcal{K}(\gamma)\cdot \int_{0}^{\ell}J^{K}(u) \,du \quad .
\end{aligned}
\end{equation}
\end{corollary}

\begin{remark}
As already alluded to in the introduction we shall encounter tractrices that are only piecewise regular. Singular cusp points typically appear in connection with pushing tractors as displayed e.g. in figures \ref{fig2DTractorClassic1}, \ref{figAreaSweep1}, and \ref{figSphTrack1}. In fact, already the tractrix in the plane can be just a single point. This happens when the tractor is a circle whose radius is equal to the wagon pole length $\ell$. Correspondingly the curvature is only piecewise defined. The contribution to the total curvature $\mathcal{K}(\gamma)$ at a singular point is defined in the usual way via the turning angle of the tangent wagon pole. In the above circle case this is then to be interpreted further as the limit value of the turning angle of the wagon pole, i.e. $2\pi$. With direct reference to corollary \ref{corSecDown}, we obtain in this particular case the area of the circular disk by that general tractor/tractrix formula for the area.
\end{remark}

\begin{remark} \label{remAvsL}
For example, with sufficient information about the Gauss curvature of the surface in figure \ref{figWildTrack4} we can estimate the total curvature of the tractrix curve from the sweep area, and with the given length of the tractor we can then estimate also the length of the tractrix.
\end{remark}

\subsection{An application to closed curves in $\mathbb{R}^{3}$} \label{subsecFenchelMilnor}

In particular we then have the following consequence of the classical results by Fenchel and Milnor concerning the total curvature of space curves -- see \cite{fenchel1929a}, \cite{milnor1953a}, \cite{Milnor1950}, and  \cite{Sullivan2008}:

\begin{corollary}
Suppose that a given tractrix curve in $\mathbb{R}^3$ is a smooth closed curve, so that the tractor curve is itself also necessarily closed. Then the area $\mathcal{A}$ of the corresponding pole sweep satisfies:
\begin{equation}
\mathcal{A} \geq \pi\cdot \ell^{2} \quad ,
\end{equation}
with $[=]$ if and only if the tractrix is a convex planar curve,
and
\begin{equation}
\mathcal{A} >   2\pi\cdot \ell^{2}
\end{equation}
if the tractrix is a nontrivial knot.
\end{corollary}


\section{Curve shortening in homotopy classes} \label{secHomotopy}

The length shortening property obtained in corollary \ref{corLLdiffer} is useful for \emph{constructing} geodesic curves between distant points $P$ and $Q$ using small wagon pole geodesics and for finding short curves in free homotopy classes -- akin to, but even more direct than, Hilbert's direct method, which is discussed and explained thoroughly in \cite[p. 355 ff.]{Spivak4}.

\begin{definition}
A \emph{self-repeated tractor/tractrix process} between two points $P$ and $Q$ in a Riemannian manifold $(M, g)$ is already indicated in figures \ref{fig2DTractorContract1} and \ref{fig2DTractorContract2}. Suppose that $\ell < \dist_{M}(P, Q)$ and choose a wagon pole $\lambda_{1}$ from $P$ to some point $p_{1}$ with distance $\ell$ from $P$. Let $\eta_{1}$ denote any tractor curve from $p$ to $Q$ and let $\gamma_{1}$ denote the ensuing tractrix curve with a final wagon pole $\lambda_{2}$ of length $\ell$ which connects the endpoint of $\gamma_{1}$ -- call it $q_{1}$ -- geodesically to $Q$. Now reverse the r\^{o}les and let $-\gamma_{1}$ (from $q_{1}$ to $P$) be the new (second) tractor, i.e. $\eta_{2} = - \gamma_{1}$, which now produces a new tractrix $\gamma_{2}$ connecting $Q$ to some point $p_{2}$ at distance $\ell$ from $P$ (the endpoints of $\lambda_{3}$). By construction, the three curves $\lambda_{1} \cup \eta_{1}$, $\gamma_{1}\cup \lambda_{2} = \lambda_{2} \cup \eta_{3}$, and $\lambda_{3} \cup \eta_{3}$ all connect $P$ and $Q$ and belong to the same homotopy class. The process is repeated and gives a sequence of curves connecting $P$ and $Q$ in the same homotopy class:
\begin{equation}
\{ \lambda_{1} \cup \eta_{1}\, , \, \,  \cdots \, , \,\,
\lambda_{n} \cup \eta_{n}\, , \, \, , \cdots \} \quad .
\end{equation}
\end{definition}

Since $\LL(\lambda_{n+1} \cup \eta_{n+1}) < \LL(\lambda_{n} \cup\eta_{n})$ unless $\lambda_{n} \cup\eta_{n}$ is already a geodesic, we have:

\begin{proposition} \label{propGeodesic}
Let $(M, g)$ be a complete Riemannian manifold and $P$ and $Q$ two points in $M$ with $\dist_{M}(P, Q) \geq \ell$. Then the self-repeated tractor/tractrix process described above converges to a geodesic in the homotopy class of curves from $P$ to $Q$.
\end{proposition}

\begin{remark}
The reulting geodesic obtained in \ref{propGeodesic} is, of course, not necessarily the shortest geodesic in the homotopy class. There may in general exist several geodesics between two points -- as on the surface in figure \ref{figWildTrack1}. If such a geodesic is locally stable, then there is also a self-repeated tractor/tractrix process that will converge to it.
\end{remark}

A slight modification of the self-repeated tractor/tractrix process to a \emph{loop-repeated tractor/tractrix process} gives a length reducing sequence of curves in any given free homotopy class of a compact non-simply connected manifold $(M, g)$ -- see  \cite[p. 355 ff.]{Spivak4}:

\begin{definition}
A \emph{loop-repeated tractor/tractrix process} is defined as follows: Let $\mu$ denote an oriented smooth piecewise regular closed curve in a given homotopy class of $M$. Replace a sufficiently small segment $\sigma$ of $\mu$ -- between points $q_{1}$ and $p_{1}$ -- by a wagon pole geodesic $\lambda_{1}$ of length $\ell$. Let $\eta_{1} = \mu - \sigma$ denote the first tractor producing a first tractrix $\gamma_{1}$ from $q_{1}$ to a point $q_{2}$ which is in distance $\ell$ from $q_{1}$. Let $\eta_{2} = \gamma_{1}$ and $p_{2} = q_{1}$, and repeat the process now using the initial pole $\lambda_{2}$ (from $q_{2}$ to $p_{2}$) and a  pull/push by the tractor
$\eta_{2} = \gamma_{1}$. This produces a sequence of closed curves which, by construction, are all in the same homotopy class:
\begin{equation}
\{ \lambda_{1} \cup \eta_{1}\, , \, \,  \cdots \, , \,\,
\lambda_{n} \cup \eta_{n}\, , \, \, , \cdots \} \quad .
\end{equation}
\end{definition}

Since again $\LL(\lambda_{n+1} \cup \eta_{n+1}) < \LL(\lambda_{n} \cup\eta_{n})$ unless $\lambda_{n} \cup \eta_{n}$ is already a geodesic, we have:

\begin{proposition} \label{propGeodesicFree}
Let $(M, g)$ be a compact non-simply connected complete Riemannian manifold with injectivity radius $\inj_{M} \geq \ell$. Then the loop-repeated tractor/tractrix process described above converges to a geodesic in the homotopy class defined by the initial tractor.
\end{proposition}


\section{Geodesic tractors and their tractrices in space forms}  \label{secSpaceFormPull}

We now apply space form trigonometry to study and reconstruct the exponential decays of $\dist(s)$ and $\kappa_{\gamma}(s)$ observed above in equation \eqref{eqConverg} for the classical tractor/tractrix system, and to generalize these results first  to constant curvature ambient spaces and then in section \ref{secVariableToponogov} to
general ambient spaces using Toponogov's comparison theorems.

In a space form every \emph{geodesic} tractor will produce a tractor/tractrix system which is contained in a $2$-dimensional totally geodesic submanifold of the symmetric space in question. Therefore we may assume without lack of generality that the dimension of the space form is $n=2$.

We consider the arc-length parametrization $\gamma(s)$ of the tractrix $\gamma$ during the pull by a geodesic
tractor $\eta$. The corresponding geodesic wagon pole for each $s$ is $\lambda_{s}(u)$, $u \in [0, \ell]$ from $\lambda_{s}(0) = \gamma(s)=A$ to $\lambda_{s}(\ell) = \eta(t(s))=B$, which together with $\eta$ defines a hinge where the second leg is the segment of $\eta$ from the pole point $\eta(t(s))=B$ to $\eta(\widehat{t}(s))=C$, which is the orthogonal geodesic projection of $\gamma(s)=A$ on $\eta$.

The geodesic from $\eta(\widehat{t}(s))=C$ to $\gamma(s)=A$ is then the third geodesic edge in the geodesic triangle $\bigtriangleup(A,B,C)$ with corresponding opposite edge lengths $a$, $b$, and $c$ respectively. As usual we will also denote the intrinsic \emph{angles} of the triangle by $A$, $B$, and $C$ at the corresponding vertices. Moreover we shall usually suppress the dependence on $s$ for all these data concerning the geodesic triangle.

We then have the following information from the discussion of the first variation and of the ensuing Jacobi fields $\Jv_{s}(u)$:
\begin{equation} \label{eqTriangleNotion}
\begin{aligned}
C &= \pi/2 \\
B &= \arctan\left(\kappa_{\gamma}(s)\cdot J_{s}(\ell)\right) \\
c &= \ell\\
b &= \dist_{M}(A, C) = \dist(s) \\
\cos(A)&= -\dist^{\prime}(s)\quad .
\end{aligned}
\end{equation}
The last identity in \eqref{eqTriangleNotion} follows immediately from the assumption that  $s$ is the arc-length along $\gamma$.
In particular, we note that the geodesic curvature $\kappa_{\gamma}(s) = \kappa(s)$ of the tractrix $\gamma$ appears nicely in the expression for the angle $B$, so that the trigonometric identities for the space form triangle eventually will give information about the size and decrease/increase of both $\dist(s)$ and $\kappa(s)$ during the tractor pull.

\subsection{Geodesic tractors in spheres} \label{subsecSpherical}

We consider first the right-angled spherical triangle $\bigtriangleup(A,B,C)$ in $\mathbb{S}_{k}$, $k>0$, of constant sectional curvature $K = k^2 >0$ under the assumption that $d < \ell \leq  k\cdot \pi/2$ so that the total perimeter of the triangle is $P < k\cdot 2\pi$. The triangle then satisfies the following trigonometric identitites (with $C=\pi/2$) -- see e.g. \cite{Bigalke1984}:

\begin{equation}
\begin{aligned}
\cos(k\cdot \ell) &= \cos(k\cdot d)\cdot \cos(k\cdot a) \\
\sin(k\cdot d) &= \sin(B)\cdot \sin(k \cdot \ell) \quad .
\end{aligned}
\end{equation}
In particular, the distance function $\dist(s) = d(s)$ satisfies:
\begin{equation}
\begin{aligned}
- d^{\prime}(s) &= \cos(A)\\
& = \sin(B)\cdot \cos(k\cdot a) \\
&= \frac{\sin(k\cdot d(s))}{\sin(k\cdot \ell)} \cdot \frac{\cos(k\cdot \ell)}{\cos(k\cdot d(s))} \\
&= \tan(k\cdot d(s))\cdot \cot(k\cdot \ell) \quad .
\end{aligned}
\end{equation}
The general solution to the differential equation
\begin{equation}
d^{\prime}(s) = -\tan(k\cdot d(s))\cdot \cot(k\cdot \ell)
\end{equation}
 is
\begin{equation} \label{eqSphDistEq}
d(s) = \frac{1}{k} \arcsin\left(e^{-k\cdot(C_{d}+s) \cdot \cot(k\cdot \ell)}\right)\quad ,
\end{equation}
where $C_{d}$ is a constant determined by $d(0)$.

The solution $d(s)$ in equation \eqref{eqSphDistEq} clearly
goes exponentially to $0$ as $s$ grows to infinity. In order to discuss and quantify  this phenomenon for the distance functions (and below for the curvature functions $\kappa(s)$) we introduce the \emph{leading exponent} $\Le(f)$ for a function $f$ as follows:
\begin{definition} \label{defLe}
Let $f: \mathbb{R_{+}} \mapsto \mathbb{R_{+}}$ and suppose that $f(s) \to 0$ for $s \to \infty$, then
\begin{equation}
\Le(f) = \lim_{s \to \infty}\left( \frac{1}{s}\cdot \ln(f(s)) \right) \quad.
\end{equation}
\end{definition}
In the above situation for geodesic tractors in ambient spaces $\mathbb{S}_{k}$ with constant positive curvature $K = k^2$ we therefore have immediately:
\begin{equation}
\Le_{K}(d) = -k\cdot \cot(k \cdot \ell)\quad, \quad \textrm{for $K = k^2 > 0$.}
\end{equation}

The curvature function $\kappa(s)$ for the tractrix $\gamma(s)$ in the same constant curvature $K = k^{2}$ situation as considered above is directly obtained from an accompanying differential equation as follows:

From
\begin{equation}
\begin{aligned}
\sin(k \cdot d) &= \sin(B)\cdot \sin(k \cdot\ell) \\
d^{\prime}(s) &= -\cos(A) = -\sin(B)\cdot\frac{\cos(k \cdot \ell)}{\cos(k \cdot d)}\\
&= \sin(\arctan(\ell \cdot \kappa(s))) \\
&= \frac{\ell \cdot \kappa(s)}{\sqrt{1 + \ell^2 \cdot \kappa^2(s)}}
\end{aligned}
\end{equation}
we get
\begin{equation}
\begin{aligned}
d^{\prime}(s)\cdot \cos(k \cdot d) &= \sin(k \cdot \ell)\cdot \frac{d}{ds}\sin(B)\\
-\sin(B)\cdot \cos(k \cdot \ell) &= \sin(k \cdot \ell)\cdot \frac{d}{ds}\sin(B) \quad,
\end{aligned}
\end{equation}
so that
\begin{equation}
\frac{d}{ds}\sin(B) = -\sin(B)\cdot \cot(k \cdot \ell) \quad .
\end{equation}
The general solution to the latter differential equation is:
\begin{equation}
\sin(B) = C_{\kappa}\cdot e^{-k\cdot s\cdot \cot(k\cdot \ell)}\quad ,
\end{equation}
where again $C_{\kappa}$ is an integration constant.
Since we also have
\begin{equation}
B = \arctan\left(\kappa(s)\cdot J_{s}(\ell)\right)
\end{equation}
and since the Jacobi field in $\mathbb{S}_{k}$ is
\begin{equation}
J_{s}(u) = \frac{1}{k}\cdot \sin(k\cdot u) \quad ,
\end{equation}
we get
\begin{equation}
\sin(B) = \frac{\kappa(s) \cdot \sin(k \cdot \ell)}{\sqrt{k^{2} + \kappa^{2}(s)\cdot \sin^{2}(k\cdot \ell)}} \quad ,
\end{equation}
and thence
\begin{equation} \label{eqKappaSph}
\kappa(s) = \frac{k\cdot C_{\kappa}\cdot e^{-k\cdot s\cdot \cot(k \cdot \ell)}}{\sin(k \cdot \ell)\cdot \sqrt{1 - C_{\kappa}^{2}\cdot e^{-2k \cdot s \cdot \cot(k \cdot \ell)}}} \quad ,
\end{equation}
where  $C_{\kappa}$ is determined e.g. by the initial value $\kappa(0)$.
This curvature function also converges to $0$ exponentially when $s$ increases. The leading asymptotic exponent is the same as for the distance function:
\begin{equation}
\Le_{K}(\kappa) = -k\cdot \cot(k \cdot \ell) \quad, \quad \textrm{for $K = k^2 > 0$.}
\end{equation}

\subsection{Geodesic tractors in hyperbolic spaces} \label{subsecHyperbolic}

We consider now the tractor/tractrix induced right angled  \emph{hyperbolic geodesic triangle} in $\mathbb{H}_{k}$ of constant curvature $K = -k^2 < 0$, $k >0$, which -- with the same notations for the triangle as above -- satisfies the hyperbolic versions of the spherical trigonometric identities -- see e.g. \cite{Fenchel1989}.

It is straightforward to mimic the equations above and obtain the corresponding hyperbolic versions of the distance function $\dist(s)$ and the curvature function $\kappa(s)$:
\begin{equation} \label{eqDistHyp}
d(s) = \frac{1}{k} \arsinh \left(e^{-k\cdot (C_{d}+s) \cdot \coth(k\cdot \ell)}\right)\quad .
\end{equation}

\begin{equation} \label{eqKappaHyp}
\kappa(s) = \frac{k\cdot C_{\kappa}\cdot e^{-k\cdot s\cdot \coth(k \cdot \ell)}}{\sinh(k \cdot \ell)\cdot \sqrt{1 - C_{\kappa}^{2}\cdot e^{-2k \cdot s \cdot \coth(k \cdot \ell)}}} \quad .
\end{equation}

\begin{equation}
\Le_{K}(d) = \Le_{K}(\kappa) = -k\cdot \coth(k \cdot \ell) \quad, \quad \textrm{for $K = -k^2 < 0$.}
\end{equation}

\subsection{Geodesic tractors in $\mathbb{R}^{n}$} \label{subsecFlat}
The formulas for $\dist_{K}(s)$ with given $\dist_{K}(0) = d(0) < \ell$ and for $\kappa_{K}(s)$ with given $\kappa_{K}(0) = \kappa(0)$, respectively, for straight line tractors in a space forms of constant curvature $K$ give the well known limit formulas in  $\mathbb{R}^{n}$ when $K \to 0^{\pm}$ -- see figures \ref{figExponents}, \ref{figDist} and \ref{figDistKappa}:

\begin{equation}
\lim_{K \to 0^{\pm}}\dist_{K}(s) = \dist_{0}(s) = d(0) \cdot e^{-s/\ell}
\end{equation}
and
\begin{equation}
\lim_{K \to 0^{\pm}}\kappa_{K}(s) = \kappa_{0}(s) =  \frac{ \kappa(0)\cdot e^{-s/\ell}}{\sqrt{1+ \ell^{2}\cdot \kappa^{2}(0) - \ell^{2}\cdot \kappa^{2}(0)\cdot e^{-2\cdot s/\ell}}}\quad .
\end{equation}

In particular, for $\kappa_{0}(0) = \infty$ (or, equivalently $d_{0}(0) = \ell$), we recover the previously found formulas for the classical tractrix in \eqref{eqConverg}, and for all the solutions we recover:

\begin{equation}
\Le_{0}(d) = \Le_{0}(\kappa) = -1/\ell \quad, \quad \textrm{for $K = 0$.}
\end{equation}

\subsection{Curvature-distance relations in space forms} \label{subsecCompSpaceformKapDist}
As already alluded to above, when we perform a geodesic tractor pull/push then the accompanying geodesic triangle contains information about
both the distance $\dist_{K}(s)$ and the curvature $\kappa_{K}(s)$ during the pull (and push) of a pole of length $\ell \leq \pi/2$ along a geodesic tractor $\eta$. These identities can be used to get explicit values of $C_{d}$ and $C_{\kappa}$ from  given values of $\kappa_{K}(0)$ and/or $\dist_{K}(0)$ in the formulas \eqref{eqSphDistEq}, \eqref{eqKappaSph}, \eqref{eqDistHyp}, and \eqref{eqKappaHyp} above. The explicit relations are as follows:

\begin{proposition}
In an ambient space with constant curvature $K = \pm k^{2}$, $k \geq 0$, we have for any geodesic tractor system:
\begin{equation} \label{eqKappaFromDist}
\kappa_{K}(s) =   \left\{
                    \begin{array}{ll}
                      \frac{k\cdot \sin(k\cdot d_{K}(s))}{\sin(k\cdot \ell) \cdot \sqrt{\cos^{2}(k \cdot d_{K}(s)) - \cos^{2}(k\cdot \ell)}} & \hbox{for $K = k^{2} > 0$ } \\
                     \frac{d_{K}(s)}{\ell \cdot \sqrt{\ell^{2} - d_{K}^{2}(s)}} & \hbox{for $K = 0$ } \\
                      \frac{k\cdot \sinh(k\cdot d_{K}(s))}{\sinh(k\cdot \ell) \cdot \sqrt{\cosh^{2}(k\cdot \ell) - \cosh^{2}(k \cdot d_{K}(s))}} & \hbox{for $K = -k^{2} < 0$.}
                    \end{array}
                  \right.
\end{equation}

\end{proposition}
\begin{proof}
Follows from the respective trigonometric identities and from the fact that the angle $B$ satisfies:
$\tan(B) = \kappa_{K}(s)\cdot J_{s}(\ell)$.
\end{proof}

\begin{figure}[h!]
\begin{center}
\includegraphics[width=90mm]{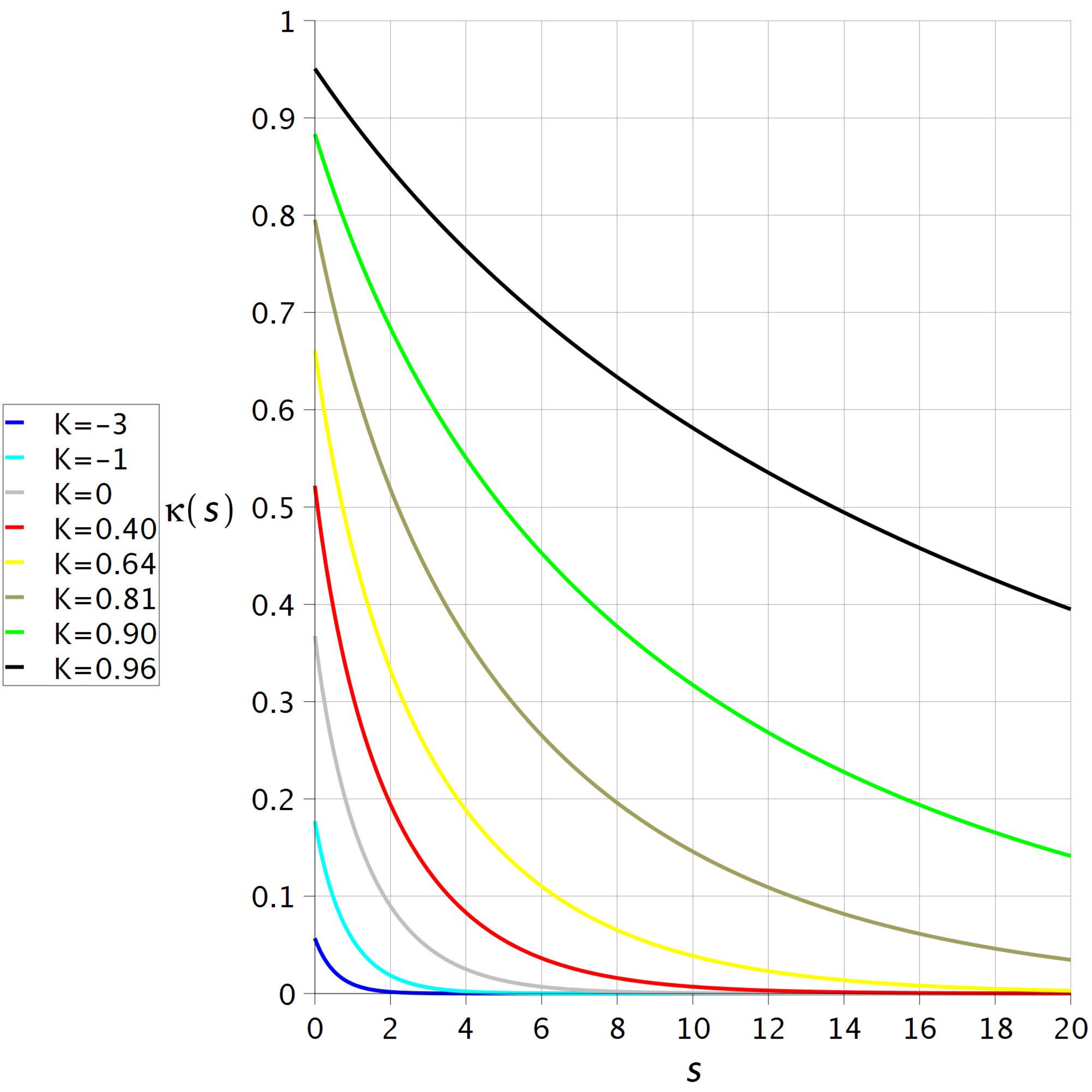}
\caption{The curvature $\kappa_{K}(s)$ of the tractrix in constant curvature $K<1$ with pole length $\ell = \pi/2$ and common initial distance $\dist_{K}(0) =  \ell/2 = \pi/4 = 0.79$.} \label{figDistKappa}
\end{center}
\end{figure}

\subsection{Leading exponent comparison} \label{subsecCompSpaceform}

\begin{figure}[h!]
\begin{center}
\includegraphics[width=100mm]{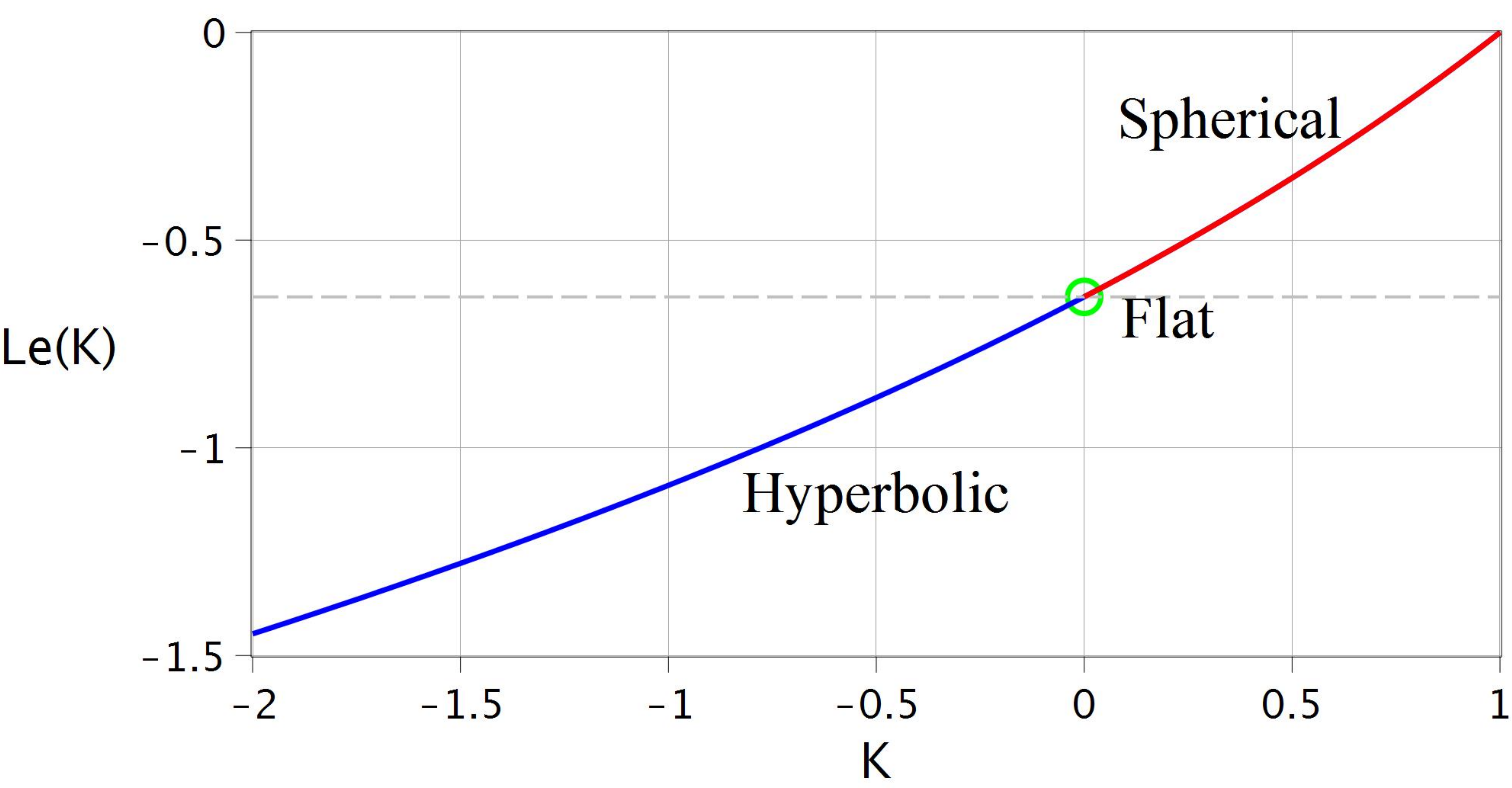}
\caption{Leading exponents $Le(K)$ for $\dist_{K}(s)$ and $\kappa_{K}(s)$ for tractrices obtained from geodesic tractors and pole length $\ell = \pi/2$ in space forms of constant curvature $K=k^{2}<1$, $K=0$ and $K=-k^{2}$, respectively.} \label{figExponents}
\end{center}
\end{figure}

With the explicit distance functions $\dist_{K}(s)$ and curvature functions $\kappa_{K}(s)$ at hand for any constant  ambient curvature $K$, we can now compare them easily as follows:

\begin{proposition} \label{propDistComp}
Suppose $\dist(0) = d_{0} < \ell$ is given. Then we have for all $s>0$ and $K_{1} < K_{2}$:
\begin{equation}
\begin{aligned}
\dist_{K_{1}}(s) &< \dist_{K_{2}}(s) \\
\kappa_{K_{1}}(s) &< \kappa_{K_{2}}(s)
\end{aligned}
\end{equation}
\end{proposition}

Moreover, for the leading coefficients we have correspondingly:

\begin{proposition} \label{propLeComp}
The leading exponents $\Le_{K}(\kappa)$ and $\Le_{K}(\dist)$ are identical for any given $K$ and satisfies:
\begin{equation}
\Le_{K} = \Le_{K}(\kappa) = \Le_{K}(\dist)
\end{equation}
as follows
\begin{equation}
\Le_{K} = \left\{
                                               \begin{array}{ll}
                                                -\sqrt{K}\cdot \cot(\ell\cdot \sqrt{K}), & \hbox{for $K>0$} \\ \\
                                                 -1/\ell, & \hbox{for $K=0$} \\ \\
                                                  -\sqrt{-K}\cdot \coth(\ell\cdot \sqrt{-K}), & \hbox{for $K<0$}
                                               \end{array}
                                             \right. \quad ,
\end{equation}
so that in particular for $K_{1} < K_{2}$ we have -- see figure \ref{figExponents}:
\begin{equation}
\Le_{K_{1}} < \Le_{K_{2}} \quad .
\end{equation}
\end{proposition}

\begin{figure}[h!]
\begin{center}
\includegraphics[width=90mm]{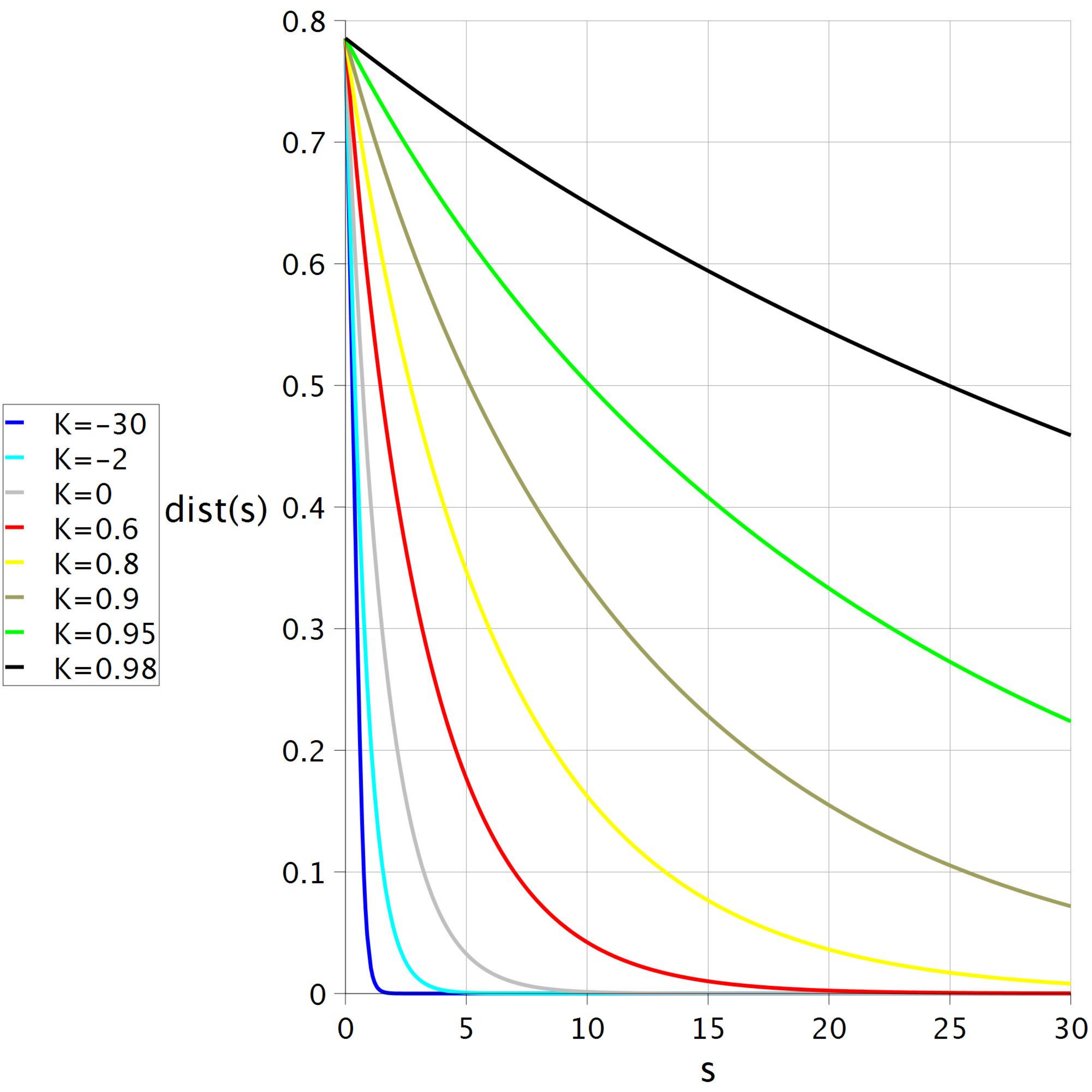}
\caption{The distance $\dist_{K}(s)$ from tractrix end point to geodesic tractor in constant curvature $K<1$ with pole length $\ell = \pi/2$ and common initial distance $\dist_{K}(0) =  \pi/4 = 0.79$ -- see propositions \ref{propDistComp} and \ref{propLeComp}.} \label{figDist}
\end{center}
\end{figure}


\section{Variable ambient curvatures\\and Toponogov's comparison theorems} \label{secVariableToponogov}

With the above observations in mind we now show, that in any  given complete Riemannian manifold with just a suitable bound on its sectional curvatures we obtain similar comparison results for tractor/tractrix systems with geodesic tractors. We apply Toponogov's triangle comparison theorems to show explicitly how the exponential decays of $\dist_{M}(s)$ and $\kappa_{M}(s)$ are controlled by the ambient sectional curvatures.

Let $(M, g)$ denote a complete Riemannian manifold and consider a tractor/tractrix system in $M$, with a geodesic tractor $\eta$ and an arclength parametrized tractrix $\gamma(s)$.

For each $s$ we consider the accompanying geodesic triangle $\bigtriangleup(A,B,C)$
in $M$ defined as in the space form setting:
The geodesic wagon pole $\lambda$ from $\gamma(s)=A$ to $\eta(t(s))=B$, defines  together with $\eta$ a geodesic hinge where the second leg is the segment of $\eta$ from the pole point $\eta(t(s))=B$ to $\eta(\widehat{t}(s))=C$, which is the orthogonal geodesic projection of $\gamma(s)=A$ on $\eta$. The minimal geodesic from $\eta(\widehat{t}(s))=C$ to $\gamma(s)=A$ is then the third geodesic edge in the geodesic triangle $\bigtriangleup(A,B,C)$ with the corresponding opposite edge lengths $a$, $b$, and $c$ respectively. Again we suppress the dependence on $s$ unless it is absolutely needed in the expressions. Again we follow the classical practice and denote the internal angles of the triangle also by $A$, $B$ and $C$.\\

We will assume throughout in this section that the conditions for applying Toponogov's triangle comparison theorems on $\bigtriangleup(A,B,C)$ are satisfied -- see
\cite[Section 4, pp. 197 ff.]{Karcher1989}, \cite{Sakai}, \cite{CheegerEbin}, and \cite{Grove2006} for further ramifications of these celebrated theorems. Specifically we assume that the edges of $\bigtriangleup(A,B,C)$ never meet the cut loci of their respective vertices, so that the above construction of $\bigtriangleup(A,B,C)$ is unique. Then we have:

\begin{theorem}[Toponogov I for upper curvature bound] \label{thmTopI}
Suppose that the sectional curvatures of $(M, g)$ satisfy $\sec_{M} < K$ for some constant $K \in \mathbb{R}$. Suppose that $\bigtriangleup(A,B,C)$ is sufficiently small in the injectivity sense alluded to above and that, if $K = k^{2} > 0$, we also have $k\cdot d <  k\cdot \ell < \pi/2$, so that the total circumference $P$ of $\bigtriangleup(A,B,C)$ is $P < 2\pi/k$. Then there exists a geodesic triangle (an Aleksandrov comparison  triangle) $\widehat{\bigtriangleup}(\widehat{A},\widehat{B},\widehat{C})$ in the space form
of constant curvature $K$ which has the same edge lengths as $\bigtriangleup(A,B,C)$, i.e. $(a,b,c) = (\widehat{a}, \widehat{b}, \widehat{c})$, and this Aleksandrov triangle satisfies the angle inequalities:
\begin{equation}
A < \widehat{A}\quad , \quad  B < \widehat{B}\quad  and \quad  \pi/2 = C < \widehat{C} \quad .
\end{equation}
\end{theorem}

\begin{theorem}[Toponogov II for lower curvature bound] \label{thmTopII}
Suppose that the sectional curvatures of $(M, g)$ satisfy $\sec_{M} > K$ for some constant $K \in \mathbb{R}$. Suppose that $\bigtriangleup(A,B,C)$ is sufficiently small in the injectivity sense explained above. Then there exists a geodesic Aleksandrov comparison  triangle $\widehat{\bigtriangleup}(\widehat{A},\widehat{B},\widehat{C})$ in the space form
of constant curvature $K$ which has the same edge lengths as $\bigtriangleup(A,B,C)$, i.e. $(a,b,c) = (\widehat{a}, \widehat{b}, \widehat{c})$,  and this Aleksandrov triangle satisfies the angle inequalities:
\begin{equation}
A > \widehat{A}\quad , \quad  B > \widehat{B}\quad  and \quad  \pi/2 = C > \widehat{C} \quad .
\end{equation}
\end{theorem}

We can now prove the following bounds on $\dist_{M}(s)$ and on $\kappa_{M}(s)$ in comparison with the space form functions $\dist_{K}(s)$ and $\kappa_{K}(s)$:

\begin{theorem} \label{thmConvUpperBd}
Suppose that the sectional curvatures of $(M, g)$ satisfy $\sec_{M} < K$ for some constant $K \in \mathbb{R}$. Let $\eta(t(s))$ and $\gamma(s)$ be a tractor/tractrix system in $M$, with a \emph{pulling} geodesic tractor $\eta(t(s))$, pole length $\ell$ and initial distance to the tractor $\dist_{M}(0) = d(0) < \ell$. Suppose that the family of accompanying geodesic triangles $\bigtriangleup(A(s),B(s),C=\pi/2)$ satisfy the conditions in theorem \ref{thmTopI} for all $s > 0$. Suppose
that $\dist_{K}(0) = d(0)$ for a corresponding tractor/tractrix system with pole length $\ell$ in the space form of constant curvature $K$ as defined and discussed in section \ref{secSpaceFormPull}. Then
\begin{equation} \label{eqDistUpperComp}
\begin{aligned}
\dist_{M}(s) &< \dist_{K}(s) \quad \textrm{and}\\
\kappa_{M}(s) & < \kappa_{K}(s) \quad \textrm{for all} \quad s \in \mathbb{R}_{+} \quad .
\end{aligned}
\end{equation}
\end{theorem}
\begin{proof}
We only need to show that if $\dist_{M}(s_{0}) = \dist_{K}(s_{0})$ for some $s_{0} \in \mathbb{R}_{0} \cup \{0\}$  -- as we assume already for $s_{0} =0$ -- then the $s$-derivatives of the distance functions satisfy:
\begin{equation} \label{eqdistPrimeUpComp1}
\dist_{M}^{\prime}(s_{0}) < \dist_{K}^{\prime}(s_{0}) \quad.
\end{equation}

The inequality \eqref{eqDistUpperComp} then follows from the smoothness of the distance functions involved. Indeed, if
$\dist_{M}(s) <  \dist_{K}(s)$ for all $0 < s < s_{1}$ and  $\dist_{M}(s) \geq  \dist_{K}(s)$ for some  $ s \geq  s_{1}$ then
$\dist_{M}(s_{0}) = \dist_{K}(s_{0})$ for some $s_{0} \geq s_{1}$ and at this value of $s_{0}$ we would have $\dist_{M}^{\prime}(s_{0}) \geq \dist_{K}^{\prime}(s_{0})$ which contradicts \eqref{eqdistPrimeUpComp1}.
The equation \eqref{eqdistPrimeUpComp1} is equivalent to the following inequality for the corresponding $A$-angles in the accompanying geodesic triangles, in $M$ and in the space form, respectively:
\begin{equation}
\begin{aligned}
\cos(A(s_{0})) &> \cos(\widetilde{A}(s_{0})) \\
A(s_{0}) &< \widetilde{A}(s_{0}) \quad .
\end{aligned}
\end{equation}
But the latter equation follows from theorem \ref{thmTopI}:
\begin{equation}
A(s_{0}) <  \widehat{A}(s_{0}) <  \widetilde{A}(s_{0}) \quad ,
\end{equation}
where the last inequality stems directly from the monotonous \emph{decrease} of $\widehat{C} > \pi/2$ to $ \widetilde{C} = \pi/2$  obtained by \emph{increasing} $\widehat{A}(s_{0})$ to $\widetilde{A}(s_{0})$ -- using the trigonometric identities in the space form of constant curvature $K$ -- and keeping edge lengths $\widetilde{c} = \ell$ and $\widetilde{a} = \dist_{M}(s)$ constant.\\
To show the curvature comparison statement in equation \eqref{eqDistUpperComp} we now apply the distance comparison result as follows: Since $\widetilde{b} = \dist_{M}(s) < \dist_{K}(s)$ we can increase $\widetilde{b}$ to $\bar{b} = \dist_{K}(s)$ in the space form while keeping $\widetilde{c} = \bar{c} = \ell$ and $\widetilde{C} = \bar{C} = \pi/2$ constants. This new space form triangle $\bigtriangleup(\bar{A}(s), \bar{B}(s), \bar{C}=\pi/2)$ will then be precisely the accompanying geodesic triangle for the tractor/tractrix system in the space form corresponding to the parameter value $s$. In the process of increasing $\widetilde{b}$ to $\bar{b}$ the angle $\widetilde{B}$ increases to $\bar{B}$, so that we get:
\begin{equation} \label{eqBtoBbar}
B < \widetilde{B} < \bar{B} \quad ,
\end{equation}
which implies that
\begin{equation}
\tan(B) = J_{s}(\ell) \cdot \kappa_{M}(s)  < \tan(\bar{B}) = J^{K}(\ell) \cdot \kappa_{K}(s) \quad ,
\end{equation}
and since $K$ is the upper bound for sectional curvatures in $M$, the Rauch comparison theorem gives $J_{s}(\ell) > J^{K}(\ell) > 0$ so that $\kappa_{M}(s) < \kappa_{K}(s)$.
\end{proof}

\begin{theorem}\label{thmConvLowerBd}
Suppose that the sectional curvatures of $(M, g)$ satisfy $\sec_{M} > K$ for some constant $K \in \mathbb{R}$. As above we let $\eta(t(s))$ and $\gamma(s)$ be a tractor/tractrix system in $M$, with a \emph{pulling} geodesic tractor $\eta(t(s))$, pole length $\ell$ and initial distance to the tractor $\dist_{M}(0) = d(0) < \ell$. Suppose that the family of accompanying geodesic triangles $\bigtriangleup(A(s),B(s),C=\pi/2)$ satisfy the conditions in theorem \ref{thmTopII} for all $s > 0$. Suppose
that $\dist_{K}(0) = d(0)$ for a corresponding tractor/tractrix system with pole length $\ell$ in the space form of constant curvature $K$. Then
\begin{equation}
\begin{aligned}
\dist_{M}(s) &> \dist_{K}(s) \quad \textrm{and}\\
\kappa_{M}(s) &> \kappa_{K}(s) \quad \textrm{for all} \quad s \in \mathbb{R}_{+} \quad .
\end{aligned}
\end{equation}
\end{theorem}
\begin{proof}
The proof mimics the proof of theorem \ref{thmConvUpperBd} modulo reversing the inequalities:
Assuming $\dist_{M}(s_{0}) > \dist_{K}(s_{0})$ for some $s_{0}$, we must now show:
\begin{equation} \label{eqdistPrimeUpComp}
\dist_{M}^{\prime}(s_{0}) > \dist_{K}^{\prime}(s_{0}) \quad,
\end{equation}
which amounts to
\begin{equation}
A(s_{0}) > \widetilde{A}(s_{0}) \quad .
\end{equation}
But the latter equation follows now from theorem \ref{thmTopII}:
\begin{equation}
A(s_{0}) >  \widehat{A}(s_{0}) >  \widetilde{A}(s_{0}) \quad ,
\end{equation}
where the last inequality stems directly from the monotonous \emph{increase} of $\widehat{C} < \pi/2$ to $ \widetilde{C} = \pi/2$  obtained by \emph{decreasing} $\widehat{A}(s_{0})$ to $\widetilde{A}(s_{0})$ -- using again the trigonometric identities in the space form of constant curvature $K$. The distance inequality then again produces the curvature inequality -- now via the opposite inequalities:
\begin{equation} \label{eqBtoBbar}
B > \widetilde{B} > \bar{B} \quad ,
\end{equation}
which implies that
\begin{equation}
\tan(B) = J_{s}(\ell) \cdot \kappa_{M}(s)  > \tan(\bar{B}) = J^{K}(\ell) \cdot \kappa_{K}(s) \quad ,
\end{equation}
and since $K$ is now the lower  bound for sectional curvatures in $M$, the corresponding Rauch comparison theorem gives $0 < J_{s}(\ell) < J^{K}(\ell)$ so that $\kappa_{M}(s) > \kappa_{K}(s)$.
\end{proof}

We finally then have the following further consequence of the results obtained above.

\begin{corollary}
Suppose the sectional curvatures of $(M, g)$ satisfy $K_{1} < \sec_{M} < K_{2}$ for some constants $K_{1} < K_{2}$ and let $\eta$ and $\gamma$ denote a tractor/tractrix system with $\eta$ a pulling geodesic tractor and pole-length $\ell$. (We assume  $\sqrt{K_{2}}\cdot \ell < \pi/2$ if $K_{2} > 0$.) Then the leading coefficients for the distance functions $\dist(s)$ and for the curvature functions $\kappa(s)$ satisfy the following inequalities in comparison with the similarly initiated tractor/tractrix systems in the space forms of constant curvature $K_{1}$ and $K_{2}$, respectively:
\begin{equation} \label{eqLeadingDoubleBounded}
\begin{aligned}
\Le(\dist_{K_{1}}) &\leq \Le(\dist_{M}) \leq \Le(\dist_{K_{2}}) \\
\Le(\kappa_{K_{1}}) &\leq \Le(\kappa_{M}) \leq \Le(\kappa_{K_{2}}) \quad .
\end{aligned}
\end{equation}
\end{corollary}

\section{Long poles on the sphere} \label{secLongPoles}

From the tractor/tractrix system on the sphere displayed in figures \ref{figSphTrack1}--\ref{figSphTrack3} it is seen, that the \emph{pull} with a \emph{long pole} of length $\ell > \pi/2$, $\ell < \pi$, just corresponds to a push (from the other diametrically opposite side) with a \emph{short} pole length of complementary size $\pi - \ell$.
 The cusp singularities of the tractrices appear naturally in this way on a sphere as they do in the cases of planar and  hyperbolic geometry:

 \begin{proposition} \label{propLongPoles}
 A unique cusp will appear for push/pull of any wagon pole along any geodesic in any space form -- except when $K = k^2 > 0$ and $\ell = q\cdot k \cdot  \pi/2$ in which cases we get: For $q$ odd: Any spherical circle (including the geodesic tractor itself)  parallel to the geodesic tractor can appear as a tractrix curve; For $q$ even: Only the geodesic tractor curve itself can appear as a tractrix.
 \end{proposition}

For the analysis -- and display -- of the spherical setting with a long pole we need the following information concerning the accompanying geodesic triangle $\bigtriangleup(A,B,C)$ and its corresponding edge lengths $(a,b,c)$:
Using equation \eqref{eqSphDistEq} with $k=1$ we get
\begin{equation} \label{eqSphDistEq2}
d(s) = \arcsin\left(e^{-(C_{d}+s) \cdot \cot(\ell)}\right)\quad ,
\end{equation}
so that with $d(0) < \ell$ given, we get the integration constant $C_{d}$:
\begin{equation}
C_{d} = - \frac{\ln(\sin(d(0))}{\cot(\ell)}
\end{equation}

From the spherical relation:
\begin{equation}
\cos(\ell) = \cos(a(s))\cdot \cos(d(s))
\end{equation}
we get
\begin{equation}
a(s) = \arccos\left(\frac{\cos(\ell)}{\cos(d(s))} \right) \quad ,
\end{equation}
which is the distance between the vertices $A(s)$ and $C(s)$ in the accompanying geodesic triangle. Finally we therefore only need to know how the accompanying triangle is moved along the tractor $\eta(t(s))$ (where now $t$ is arclength along $\eta$) when $A(s)$ is moving along the tractrix with unit speed parametrized by arclength $s$. But since
\begin{equation}
\frac{d}{ds}t(s) = \tan(B(s)) = \frac{\tan(d(s))}{ \sin(a(s))}\quad ,
\end{equation}
we get
\begin{equation}
t(s) = a(0) + \int_{0}^{s} \frac{\tan(d(u))}{ \sin(a(u))} \, du \quad ,
\end{equation}
so that the tractor is located at $B(s) = \eta(t(s))$ when the tractrix point $A(s)$ has
distance $d(s)$ from $C(s) =\eta(t(s) - a(s))$. The accompanying geodesic triangle can now be constructed directly in spherical coordinates from these ingredients -- see figures \ref{figSphTrack1}--\ref{figSphTrack3}.
The pushing scenario with pole lengt  $ \pi - \ell$ is briefly discussed in the caption of figure \ref{figSphTrack3}.

\begin{figure}[h!]
\begin{center}
\includegraphics[width=50mm]{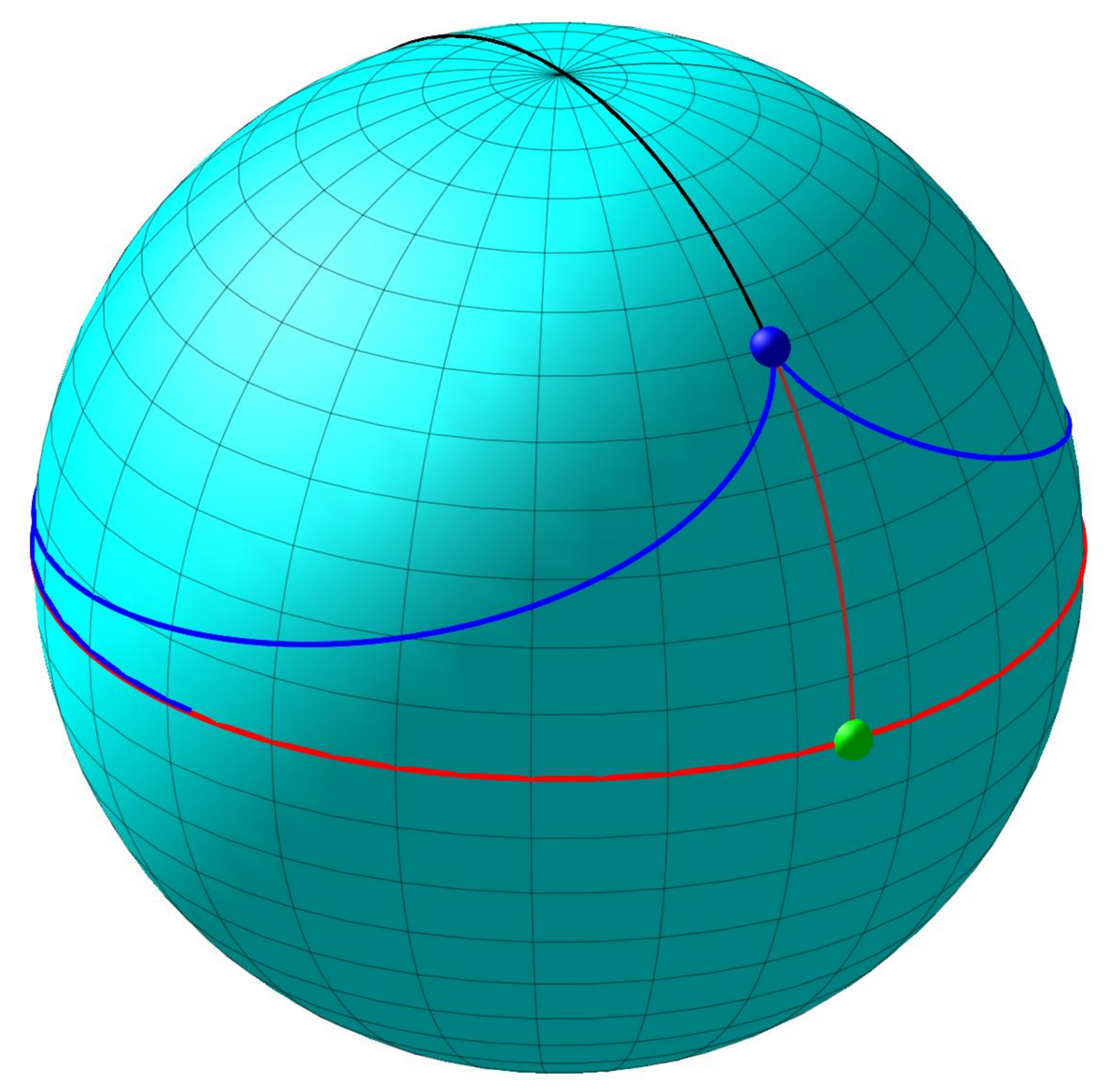}
\caption{A geodesic tractor, the equator, and a typical complete tractrix with one cusp point singularity on a sphere.} \label{figSphTrack1}
\end{center}
\end{figure}

\begin{figure}[h!]
\begin{center}
\includegraphics[width=50mm]{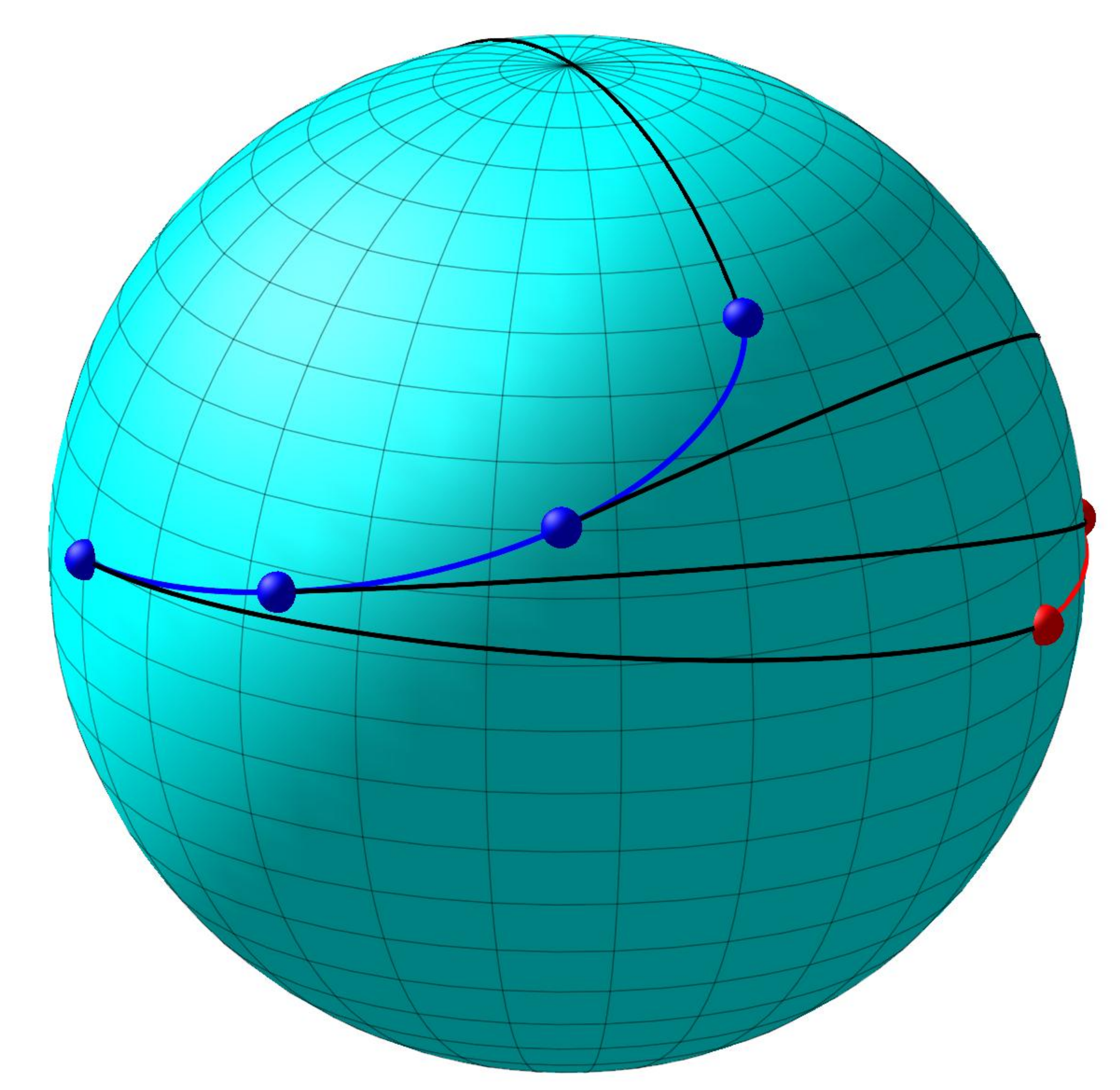} \quad \quad
\includegraphics[width=50mm]{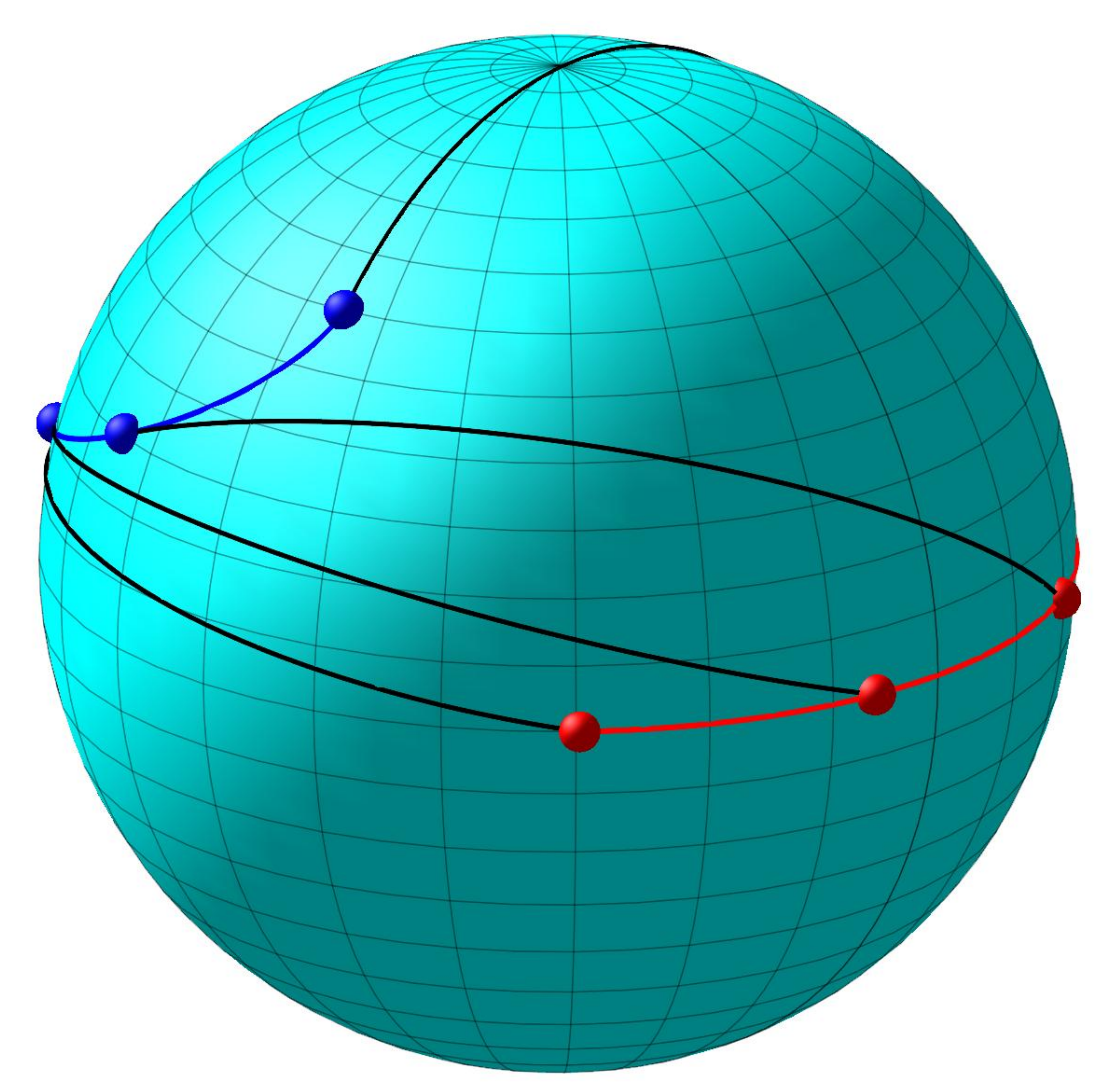}
\caption{The equatorial tractor (red) and a (blue) pulled tractrix. The pull is performed from left to right but may also be considered as a push from right to left via the black geodesic wagon pole. The length of the pole is $\ell = 3\pi/4 > \pi/2$ and the (leftmost) initial distance of the tractrix to the equator is set equal to $\dist(0) = \ell/20 = 3\pi/80$.} \label{figSphTrack2}
\end{center}
\end{figure}

\begin{figure}[h!]
\begin{center}
\includegraphics[width=50mm]{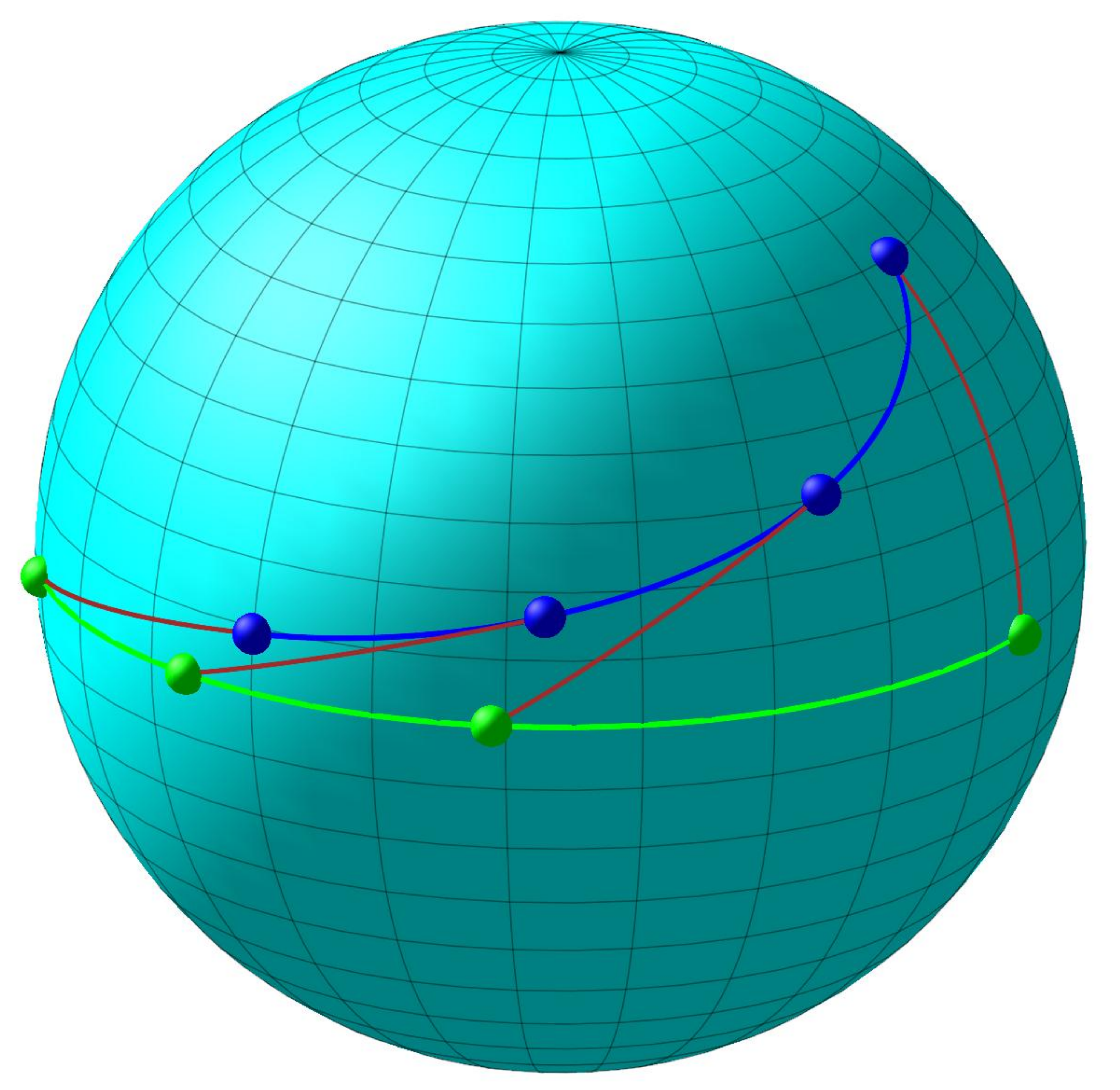} \quad \quad
\includegraphics[width=50mm]{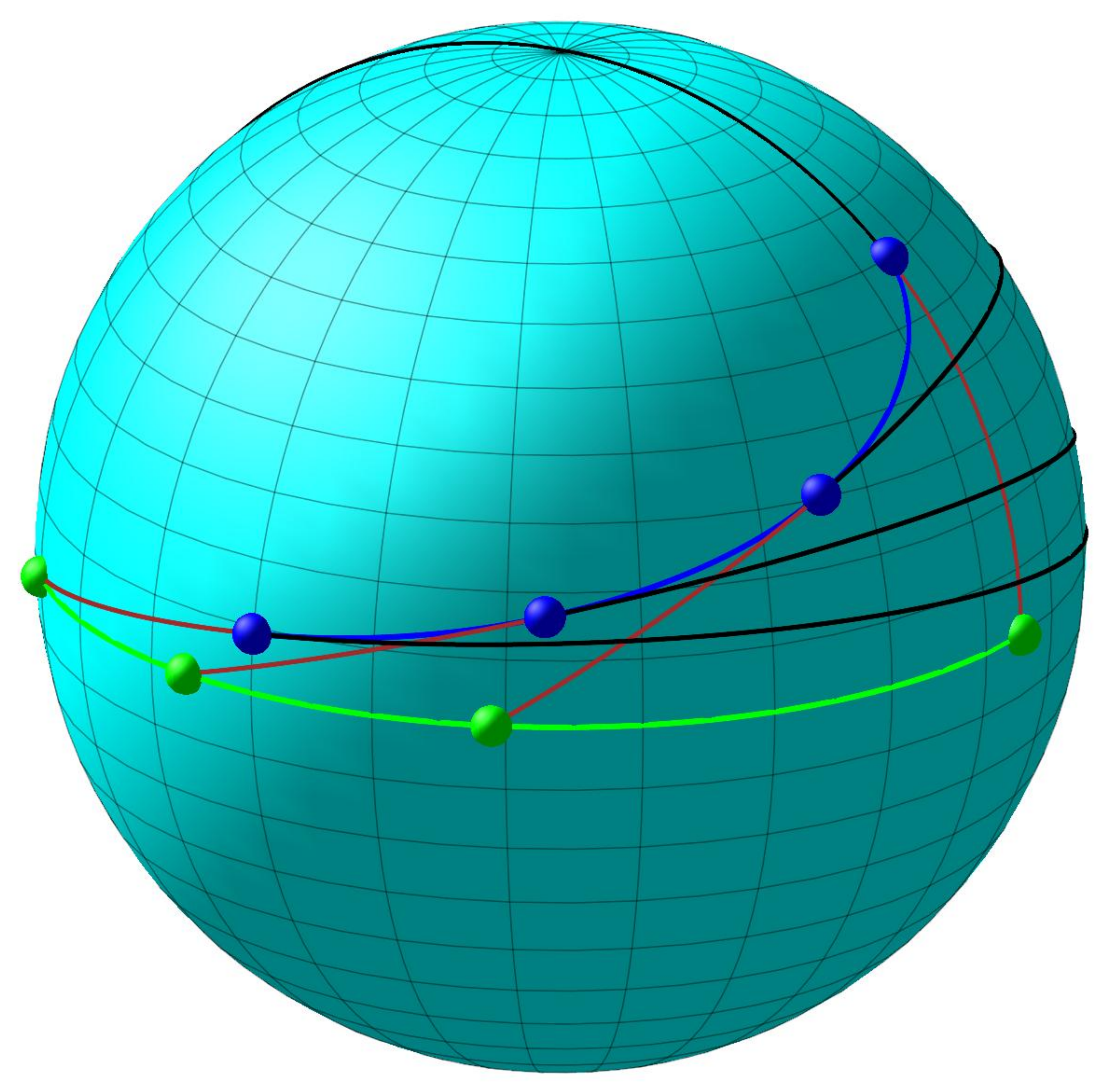}
\caption{Another equatorial tractor (green) segment and the same (blue) tractrix as i figures \ref{figSphTrack1} and \ref{figSphTrack2}. The pull is now performed from right to left but may also be considered as a push from left to right. The length of the (brownish) pole is now $\pi/4 < \pi/2$ and the (leftmost) initial distance of the tractrix to the equator is still equal to $\dist(0) = 3\pi/80$. The push/pull scenarios, the first one with the red tractor and the long black pole  and the second one with the green tractor and the short brown pole are clearly dual to each other and produce the same total tractrix shown in figure \ref{figSphTrack1}. A long pole length can be replaced by a dual short pole length due to the symmetry of the sphere.} \label{figSphTrack3}
\end{center}
\end{figure}


\section{Discussion} \label{secDiscurs}
Our discussion has so far only been concerned with the most fundamental properties of the generalized tractor/tractrix systems in Riemannian manifolds. The many interesting applications and geometric developments from such systems in the plane -- concerned with e.g. multi-trailer systems etc. -- represent straightforward invitations to study similar constructions and questions in the general setting of Riemannian manifolds with nontrivial curvature and topology.\\

We mention but one example of such an aftermath question from the present work: The observed contraction properties (of e.g. the distance $\dist_{M}$) under the pull of a tractrix along a given geodesic tractor should also be controllable 'on the average' via a suitable bound on the Ricci curvatures of the ambient manifold  -- see e.g. the seminal works on tubes by A. Gray: \cite{Gray2004}, \cite{Gray1982}, and \cite{Gray1982a}.\\

These issues will be taken up in forthcoming works by the present authors and/or by other authors.\\

\vspace{1cm}

{\bf{Acknowledgement.}} The first named author would like to thank the Carlsberg Foundation for support via the grants CF15-0552 and CF16-0639.\\


\bibliographystyle{plain}



\end{document}